\documentclass[10pt]{article}
\usepackage[utf8]{inputenc}
\usepackage{color}
\usepackage[dvipsnames]{xcolor}
\usepackage{physics}
\usepackage{algorithm}
\usepackage{algpseudocode}
\usepackage{hyperref}
\hypersetup{colorlinks,allcolors=red}
\usepackage{mathtools}
\usepackage{graphicx}
\usepackage{relsize}
\usepackage{amsthm,amsmath,amsfonts}
\usepackage{amssymb}
\usepackage{float}
\usepackage[numbers]{natbib}
\usepackage{enumerate}
\usepackage{xcolor}
\usepackage[margin=1in]{geometry}
\usepackage{fancyhdr}
\usepackage{bm}
\newtheorem{thm}{Theorem}[section]

\newtheorem{prop}[thm]{Proposition}

\usepackage{comment}

\theoremstyle{definition}
\newtheorem{defn}[thm]{Definition}

\graphicspath{{Pics/}}
\newcommand{\ud}{\,\mathrm{d}}

\newcommand{\olsi}[1]{\,\overline{\!{#1}}} % overline short italic

\theoremstyle{remark}

\title{Stochastic Delay Differential Games: Financial Modeling and Machine Learning Algorithms}

\author{Robert Balkin\thanks{Department of Mathematics, University of California, Santa Barbara, CA 93106-3080, USA, \em{rbalkin@ucsb.edu}.} \and Hector D. Ceniceros \thanks{Department of Mathematics, University of California, Santa Barbara, CA 93106-3080, USA,  \em{ceniceros@ucsb.edu}.} \and Ruimeng Hu\thanks{Department of Mathematics, and Department of Statistics and Applied Probability, University of California, Santa Barbara, CA 93106-3080, {\em rhu@ucsb.edu}.}}

\begin{document}
\maketitle

\begin{abstract}

In this paper, we propose a numerical methodology for finding the closed-loop Nash equilibrium of stochastic delay differential games through deep learning. These games are prevalent in finance and economics where multi-agent interaction and delayed effects are often desired features in a model, but are introduced at the expense of increased dimensionality of the problem. This increased dimensionality is especially significant as that arising from the number of players is coupled with the potential infinite dimensionality caused by the delay.

Our approach involves parameterizing the controls of each player using distinct recurrent neural networks. These recurrent neural network-based controls are then trained using a modified version of Brown's fictitious play, incorporating deep learning techniques. To evaluate the effectiveness of our methodology, we test it on finance-related problems with known solutions. Furthermore, we also develop new problems and derive their analytical Nash equilibrium solutions, which serve as additional benchmarks for assessing the performance of our proposed deep learning approach.
\end{abstract}

\section{Introduction}
Stochastic delay differential games combine game theory and stochastic control problems with delay. These control problems encompass various models applicable to economics, advertising, and finance. For instance, in determining a firm's optimal advertising policy, Gozzi and Marinelli \cite{advertising_models_1} consider a model which incorporates the delayed impact of advertising expenditures on the firm's goodwill. Similarly, in finance, optimal investment and consumption decisions could also take into account delayed market features as is done by Pang and Hussain \cite{delayedproblem}. Furthermore, these delayed stochastic control problems can often be extended to incorporate interaction with competitors, who can influence both the underlying system dynamics and the objectives of individual actors. In the context of such scenarios, the stochastic control problem with delay can be further extended to a stochastic delay differential game. This framework captures the interaction among all participants (or players) who select their controls to optimize their objectives. The controls of each player affect the system dynamics, which are modeled as a system of stochastic delay differential equations (SDDEs). The outcome of the game is represented by the concept of Nash equilibrium, which is a collection of all players' choices, ensuring that no player has the incentive to deviate unilaterally.

Despite introducing mathematical and computational challenges, incorporating delay is crucial for developing more complex and realistic models that capture real-world phenomena. For instance, in the analysis of systemic risk, Carmona, Fouque and Sun~\cite{MFGandSysRisk} model bank lending and borrowing as a stochastic differential game without delay, assuming a specific form of bank repayments at time $t$. To enhance the model's realism, the same authors in collaboration with Mousavi~\cite{CarmonaFouqueMousaviSun} consider banks that must repay their borrowings at time $t$ by time $t+\tau$, introducing a delayed factor into the governing dynamics. While this model effectively captures the nature of delayed repayments, it also increases the mathematical and computational complexity of the underlying problem. However, given the widespread occurrence and realistic nature of delayed problems, it is essential to address the computational challenges they present.

The primary reason for these difficulties lies in the inherent dimensionality of the problem. Stochastic differential games already face the curse of dimensionality when the number of players, denoted as $N$, is large. Adding to this inherent complexity, stochastic delay differential games introduce a possibly infinite-dimensional component, as the drift and volatility of the associated SDDE depend on the entire path. To formalize this, we note that one can employ the approach of dynamic programming to characterize the value functions associated with the closed-loop Nash equilibrium through a system of Hamilton-Jacobi-Bellman (HJB) equations, enabling the determination of Nash equilibrium controls. However, the resulting HJB equations in the delayed case involve derivatives with respect to variables in an infinite-dimensional Hilbert space as detailed in Section 2.6.8 in the book by Fabbri, Gozzi, and Swiech~\cite{SOCinfdim}. Numerically solving this HJB system would require an additional high-dimensional approximation to handle the infinite dimensionality arising from the delay. However, deep learning methodologies are natural choices for solving problems with high dimensionality and have been used in similar instances. For example, Fouque and Zhang~\cite{fouque2019deep} parameterize the optimal control with neural networks to solve a mean field control problem arising from  an inter-bank lending model with delayed repayments, and Han and Hu~\cite{RNNSCPwD} solve the stochastic control problems with delay using neural networks. 

To address the challenge of high dimensionality in these problems, we propose a deep learning-based method that effectively handles the delay.  Inspired by the approach presented in \cite{RNNSCPwD}, which utilizes recurrent neural networks (RNNs) to solve stochastic control problems with delay, we introduce an algorithm for finding the Nash equilibrium of stochastic delay differential games. Specifically, we parameterize players' controls using RNNs and approximate their objective functions by sampling the game dynamics under these RNN-based controls. The parameters of the RNNs are then optimized using the concept of deep fictitious play, as introduced in \cite{DFPSDG,DFP2}. The utilization of neural network-based control functions enables us to reformulate the problem in a finite-dimensional setting. Now, the optimization for a given player revolves around selecting the neural network parameters. Moreover, employing RNNs, in particular, allows us to effectively capture the influence of delay in players' controls, as RNNs have the capability to learn the appropriate memory dependencies present in the true Nash equilibrium controls.

Separately, we introduce a new class of problems motivated by portfolio optimization with delayed dynamics and competition among portfolio managers \cite{delayedproblem, Competition_among_Port_Managers,lackerNoConsump,lackerwconsump}. These problems arise from a model where portfolio managers trade in a financial market, considering tax consequences and being incentivized to perform well on both an absolute and relative basis. Specifically, there is a delay between when taxes are realized and when they become due. We summarize our results for these new problems in Sections \ref{subSection the competition of portfolio managers with Delayed Tax Effects} and \ref{subSection the Competitive Consumption and Portfolio Allocation with Delayed Tax Effects problem}, with full motivations provided in Appendix \ref{Appendix Intuition} and proofs in Appendix \ref{appendix_proof}.

We then validate the proposed deep learning algorithm numerically on a set of problems with known closed-form solutions. This includes the new class of problems that we introduce and solve in this paper as well as the model introduced in \cite{CarmonaFouqueMousaviSun} to study the systemic risk in bank lending. By considering both new and existing problems, we assess the accuracy of our proposed method using their closed-form solutions as benchmarks. Our numerical experiments confirm the success of our algorithm in approximating the true Nash equilibrium for all the problems considered.

The outline of the paper is as follows. In Section \ref{The_Mathematical_Problem}, we introduce the mathematical description of stochastic delay differential games and define the notion of closed-loop Nash equilibrium. In Section \ref{Section benchmark examples}, we formulate and provide analytical solutions for the closed-loop Nash equilibrium of the stochastic delay differential games we later solve numerically using our proposed algorithm. Specifically, we devote Sections \ref{subSection the competition of portfolio managers with Delayed Tax Effects} and \ref{subSection the Competitive Consumption and Portfolio Allocation with Delayed Tax Effects problem} to present the new problems and their solutions. The model construction of these problems can be found in Appendix \ref{Appendix Intuition}, while the proofs of their solutions are in Appendix \ref{appendix_proof}. In Section \ref{section the discrete problem and algorithm}, we propose a numerical algorithm for approximating stochastic delay differential games, and in Section \ref{Section Numerical results}, we demonstrate the numerical results of our algorithm compared to the known solutions of the considered problems. Finally, we provide some concluding remarks in Section~\ref{Section conclusion}.

\section{The Mathematical Problem}\label{The_Mathematical_Problem}
We define the problem mathematically following the similar setup in \cite{RNNSCPwD} for stochastic control problems with delay. The general problem we consider begins with an SDDE system, where the delay can be potentially present in both the state variables as well as the controls. Formally, on a complete probability space $(\Omega, \mathcal{F}, \mathbb{P})$, we have the $N$-player stochastic delay differential game driven by the dynamics
\begin{equation}\label{general_SDDG_dynamics}
\begin{aligned}
 & \ud \bm{X}^{\bm{\alpha}}_t = \mu (t,\bm{X}^{\bm{\alpha}}_{[t-\tau,t]}, \bm{\alpha}_{[t-\tau,t]}) \ud t + \sigma(t,\bm{X}^{\bm{\alpha}}_{[t-\tau,t]},\bm{\alpha}_{[t-\tau,t]})\mathrm{d} \bm{W}_t, \   t \in [0,T],\\
 & \bm{X}^{\bm{\alpha}}_{t} = \bm{\zeta}(t), \   t \in [-\tau,0], \\
& \bm{\alpha}_{t} = \bm{\phi}(t),  \  t \in [-\tau,0), \\
\end{aligned}
\end{equation}
where $\bm{X}^{\bm{\alpha}}$ is the state process which takes values in $\mathbb{R}^n$, and $\bm{\alpha} = (\alpha^1, \cdots, \alpha^N)$ is the collection of all players' controls. Here, $\alpha^i_t$ is the strategy or decision of player $i$ at time $t$ and takes values in the control space $\mathcal{A}^i \subset \mathbb{R}^{m_i}$. We use the notation $\bm{X}^{\bm{\alpha}}_{[t-\tau,t]}$ to represent the paths of the stochastic process $\bm{X}^{\bm{\alpha}}$ along the interval $[t-\tau,t]$, and similar notation for $\bm{\alpha}_{[t-\tau,t]}$. We call $\tau > 0$ (deterministic) the ``length of the delay'' or simply the ``delay'' as Eq.~\eqref{general_SDDG_dynamics} shows that the increment of the state process at time $t$ depends on the entire history of the state and control processes as far back as $\tau$ units in the past. The drift, $\mu$, and volatility, $\sigma$, are functionals that map into $\mathbb{R}^n$ and $\mathbb{R}^{n \times k}$ respectively, and $\bm{W}$ is a $k$-dimensional, standard Brownian motion. 

We remark that there is a special case in which one can write the state process as $\bm{X} = (X^1, \cdots, X^N)$, where $X^i$ is affected only through the control $\alpha^i$. In this instance, $X^i$ is the private state of player $i$. In our case, we generically assume $\bm{X}^{\bm{\alpha}}_t(\omega) \in \mathbb{R}^n$ and could represent a combination of both private states as well as public states -- those that are shared and influenced collectively.

 Formally, we define $\bm{X}^{\bm{\alpha}}_{[t-\tau,t]}$ to be a map from $[-\tau,0]$ to the space of square integrable random variables $L^2(\Omega)$ given by $\bm{X}^{\bm \alpha}_{[t-\tau,t]}(s) := \bm{X}^{\bm \alpha}_{s+\tau+t}$. In particular, we seek solutions $\bm{X}^{\bm{\alpha}}$ to Eq.~\eqref{general_SDDG_dynamics} such that for each $t \in [0,T]$, we have that $\bm{X}^{\bm{\alpha}}_{[t-\tau,t]} \in L^2(\Omega; C([-\tau,0] ; \mathbb{R}^n))$. Here, the space $L^2(\Omega; C([-\tau,0] ; \mathbb{R}^n))$ is defined as the normed space of $C([-\tau,0] ; \mathbb{R}^n)$ valued random variables with the norm given by
\begin{equation*}
||\bm{Z}||_{L^2(\Omega; C([-\tau,0] ; \mathbb{R}^n))} = \left(\mathbb{E}\left[ \sup_{s \in [-\tau,0]} |\bm{Z}_{s}(\omega)|^2 \right] \right)^{\frac{1}{2}}.
\end{equation*}
The stochastic path $\bm{\alpha}_{[t-\tau,t]}$ is defined analogously by ${\bm \alpha}_{[t-\tau,t]}(s) := \bm{\alpha}_{s+\tau+t}$, where the stochastic processes $(\bm{\alpha}_t)_{t \in [0,T]}$ belongs to an admissible set $\mathbb{A}$ defined later by the set definition \eqref{admissible_sets}.

For a fixed choice of controls $\bm \alpha$, one can consider the existence and uniqueness of the SDDE \eqref{general_SDDG_dynamics}. For this SDDE, one can require $\mu$ and $\sigma$ to be Lipschitz in the second argument to ensure the existence and uniqueness of a strong solution. To be precise, this Lipschitz condition is
\begin{equation*}
\begin{aligned}
   &  ||\mu(t,\bm{x_1}, \bm{\alpha}_{[t-\tau,t]}) - \mu(t,\bm{x_2}, \bm{\alpha}_{[t-\tau,t]})||_{L^2(\Omega)} \leq L ||\bm{x_1} - \bm{x_2} ||_{L^2(\Omega; C([-\tau,0] ; \mathbb{R}^n))}, \\
   &  ||\sigma(t,\bm{x_1}, \bm{\alpha}_{[t-\tau,t]}) - \sigma(t,\bm{x_2}, \bm{\alpha}_{[t-\tau,t]})||_{L^2(\Omega)} \leq L ||\bm{x_1} - \bm{x_2} ||_{L^2(\Omega; C([-\tau,0] ; \mathbb{R}^n))}, \\
\end{aligned}
\end{equation*}
for some $L>0$ and for all $t \in [0,T], \  \bm{x_1}, \bm{x_2} \in L^2(\Omega; C([-\tau,0] ; \mathbb{R}^n))$. For more details of the existence and uniqueness theory for SDDEs, we refer to the work by Mohammed~\cite{Mohammed2}.

Now that we have defined the SDDE and considered its existence and uniqueness for a given $\bm \alpha$, we proceed by specifying the requirements for $\bm \alpha$. We require that each control, $\alpha^i$, is in the class of closed-loop controls within an admissible set. Intuitively, the closed-loop controls are those that take the form $\alpha^i_t = \phi^i(t,\bm{X}^{\bm{\alpha}}_{[-\tau,t]})$, and therefore represent a decision at time $t$ based on observation of the state process up to and including this time. Formally, closed-loop controls are progressively measurable processes with respect to the filtration generated by the state process $\bm{X}^{\bm{\alpha}}_t$. Denoting $\mathcal{F}_t^{\bm{X}}$ to be the filtration generated by $\bm{X}^{\bm{\alpha}}_{[-\tau,t]}$, we define the admissible set of closed-loop controls $\mathbb{A}^i$ for player $i$ to be
\begin{equation}\label{admissible_sets}
    \mathbb{A}^i = \Big\{ \mathcal{F}_t^{\bm{X}}\mathrm{-progressively \ measurable \ processes \ } \beta^i : [-\tau,T] \times \Omega \rightarrow \mathcal{A}^i \subset \mathbb{R}^{m_i}  \bigg| \int_{-\tau}^T \mathbb{E} [|\beta^i_t|^2] \ud t < \infty \Big\},
\end{equation}
and we denote the product space of admissible controls by $\mathbb{A} = \otimes_{i=1}^N \mathbb{A}^i$ along with the control space for all players by $\mathcal{A} = \otimes_{i=1}^N \mathcal{A}^i$.

Next, we define the running and terminal costs for player $i$ as  $f^i$ and $g^i$, respectively. Here, $f^i: [0,T] \times L^2(\Omega, C([-\tau,0];\mathbb{R}^{n})) \times L^2(\Omega, C([-\tau,0];\mathcal{A})) \rightarrow \mathbb{R}$ and $g^i: L^2(\Omega, C([-\tau,0];\mathbb{R}^{n})) \rightarrow \mathbb{R}$ are deterministic measurable functionals. From these functionals, we define the expected cost  $J^i$ for player $i$ to be
\begin{equation}\label{general_SDDG_cost}
    J^i[\bm{\alpha}] = \mathbb{E}  \left[ \int_0^T f^i(t,\bm{X}^{\bm{\alpha}}_{[t-\tau,t]},\bm{\alpha}_{[t-\tau,t]})\mathrm{d}t + g^i(\bm{X}^{\bm{\alpha}}_{[T-\tau,T]})   \right].
\end{equation}
We remark that one can also consider games where $J^i$ in Eq.~\eqref{general_SDDG_cost} defines the reward of player $i$ rather than cost. In this case, one can cast the problem as one of costs by taking the cost of player $i$ to be $-J^i$. Because of this, we will be referring to problems where $J^i$ represents the cost of player $i$ unless otherwise stated.

The problem we consider is to find the Nash equilibrium in $\mathbb{A}$ for the stochastic delay differential game described in \eqref{general_SDDG_dynamics} and \eqref{general_SDDG_cost}. The Nash equilibrium is a set of controls whereupon each player has optimally chosen their own control given the choices of controls for the other players. To be precise, we understand the Nash equilibrium through the following definition.

\begin{defn}\label{NashEq}
 We say that $\bm{\alpha}^* = (\alpha^{*1} , \cdots, \alpha^{*N}) \in \mathbb{A}$ is a Nash equilibrium if  
\begin{equation}\label{general_NE}
J^i[\bm{\alpha}^*] \leq J^i[\alpha^{*1} , \cdots, \alpha^{*i-1}, \ \beta^{i}, \alpha^{*i+1},\alpha^{*N}], \quad  \forall i \in \{1 , \cdots, N \}, \, \forall \beta^{i} \in \mathbb{A}^i.
\end{equation}
\end{defn}

\noindent We remark that since $\mathbb{A}$ by definition contains the admissible, closed-loop controls, we call a Nash equilibrium $\bm{\alpha}^* \in \mathbb{A}$ a closed-loop Nash equilibrium.

In conclusion, the full problem of finding the closed-loop Nash equilibrium for a stochastic delay differential game is defined through the key components \eqref{general_SDDG_dynamics}--\eqref{general_NE}. In essence, Eq.~\eqref{admissible_sets} defines the appropriate space of controls, while Eqs.~\eqref{general_SDDG_dynamics}, \eqref{general_SDDG_cost} define a map from the choices of controls for each player $\alpha^i \in \mathbb{A}^i$ to the cost experienced by a given player $J^i$. The condition \eqref{general_NE} expresses how each player rationally chooses her strategy 
based on her own cost, given the choices of the other players, leading to equilibrium.

\section{Three Games with Delay}\label{Section benchmark examples}
In this section, we present three stochastic differential games with delay. In Section~\ref{subSection the competition of portfolio managers with Delayed Tax Effects} and \ref{subSection the Competitive Consumption and Portfolio Allocation with Delayed Tax Effects problem}, we derive two novel problems that are inspired by \cite{delayedproblem,Competition_among_Port_Managers,lackerwconsump,lackerNoConsump}. 
For clarity, we shall present the mathematical formulations and highlight analytical results below and defer modeling motivations and intuitions for these new problems to Appendix \ref{Appendix Intuition} and the proofs of their solutions  to Appendix \ref{appendix_proof}. In Section~\ref{subSection Inter-Bank Lending Model}, we briefly review a stochastic delay differential game arising from inter-bank lending, as discussed in \cite{CarmonaFouqueMousaviSun}, and summarize the results therein. All three problems will serve as benchmarks for the numerical methodology we propose in  Section \ref{section the discrete problem and algorithm}.

\subsection{Competition between Portfolio Managers with Delayed Tax Effects}\label{subSection the competition of portfolio managers with Delayed Tax Effects}

We consider a portfolio game between $N$ managers where everyone's award depends on both their absolute and relative performance, subject to delayed tax effects. Such a problem is inspired by the model problems introduced in \cite{delayedproblem} and \cite{lackerNoConsump}, and the full intuition and derivations are elaborated in Appendix \ref{Appendix Intuition}.

Let $X^i_t \in \mathbb{R}$ be the wealth at time $t$ of an investor $i$. Her wealth process is influenced by $\pi^i_t \in \mathbb{R}$, the \emph{fraction} of wealth she chooses at time $t$ to allocate into a risky asset, while the remaining is left in a money market account accruing at a risk-free rate $r \in \mathbb{R}$. At time $t$, the investor pays taxes at a rate of $\mu_2$ on her exponentially averaged past wealth $Y_t^i$. Precisely, the dynamics for the wealth of each player $i \in \{1,\cdots,N\}$ are given by
\begin{equation}\label{XdyanmicsMertonNoCons}
\begin{aligned}
  \ud X^i_t & =  \left[(\mu_1 - r) \pi^i_t X^i_t  + rX^i_t -\mu_2 Y^i_t  \right] \ud t + \sigma \pi^i_t X^i_t \ud W_t,  \ t \in (0,T], \\
    Y^i_t & = \int_{-\infty}^t \lambda e^{-\lambda(t- s)} X^i_{s} \ud s,  \ t \in (0,T], \\
  X^i_t & = \zeta^i(t), \ t \in (-\infty,0]. \\
\end{aligned}
\end{equation}
 Here, $W$ is a 1-D Brownian motion, and the initial wealth $\zeta^i$ is positive and bounded for $t \in (-\infty,0]$. The parameter $\mu_1 \in \mathbb{R}$ is the mean return of the stock with $\mu_1 > r$, and $\sigma>0$ is its volatility from the Black-Scholes model. The parameter $\lambda > 0$ is the arrival rate of tax billings as explained in Appendix \ref{appendix Portfolio Optimization with Delayed Taxes}.  We note that in this case, the length of the delay is $\tau = \infty$ as seen through the dependency on $Y^i_t$ in Eq.~\eqref{XdyanmicsMertonNoCons} which itself depends on the entire path $X^i_{(-\infty,t]}$.

 We consider two cases for the reward for player $i$. The first case is based on the constant absolute risk aversion (CARA) utility and is given by 
\begin{equation}\label{costMertonNoConsCARA1}
\begin{aligned}
    J^i[\bm \pi] = \mathbb{E} \left[U_i \left( Z^i_{disc,T} - \theta_i \olsi{Z}_{disc,T}  \right) \right], \\
\end{aligned}
\end{equation}
where $0 < \theta_i < 1$, and
\begin{equation}\label{costMertonNoConsCARA2}
 \olsi{Z}_{disc,T} = \frac{1}{N} \sum_{i=1}^N Z_{disc,T}^i, \qquad U_i(z)  = - \exp(-\frac{1}{\delta_i} z),\\
\end{equation}
with $\delta_i>0$ and $Z^i_{disc}$ defined by
\begin{equation}\label{a and Z_disc}
\begin{aligned}
 & Z^i_t  = X^i_t + aY^i_t, \qquad \mathrm{where} \ a =  \frac{-(r+\lambda) + \sqrt{(r+\lambda)^2 - 4 \lambda \mu_2}}{2\lambda}, \\
 & Z_{disc,t}^i  = e^{-(r+\lambda a)t}Z^i_t.\\
 \end{aligned}
\end{equation}
The second case is the based on the constant relative risk aversion (CRRA) utility and given by
 \begin{equation}\label{costMertonNoConsCRRA1}
\begin{aligned}
    J^i[\bm \pi] & = \mathbb{E} \left[U_i \left( Z^i_{disc,T}  \olsi{Z}_{disc,T}^{-\theta_i}  \right) \right], 
\end{aligned}
\end{equation}
where $0 < \theta_i < 1$ and
\begin{equation}\label{costMertonNoConsCRRA2}
  \olsi{Z}_{disc,T}  =  \left(\prod_{i=1}^N Z^i_{disc,T}  \right)^{1/N}, \qquad  U_i(z) = \begin{cases}  
        \frac{1}{1- \frac{1}{\delta_i}} z^{1- \frac{1}{\delta_i}}, &  \delta_i \neq 1, \\
        \log(z), &  \delta_i = 1,
    \end{cases}
\end{equation}
with $\delta_i>0$ and $Z^i_{disc}$ defined by Eq.~\eqref{a and Z_disc}. 

From the definition above, we notice that $(r+\lambda)^2 - 4 \lambda \mu_2 > 0$ must be required, which essentially means that the tax effect cannot be too large. We further require $r+\lambda > 0$, resulting in $a<0$. We also remark that $r+\lambda a$ can be ascribed the meaning of a ``tax-adjusted risk-free rate'' and $Z^i_t$ can be ascribed the ``tax-adjusted wealth''. The utilities in Eq.~\eqref{costMertonNoConsCARA1} and  Eq.~\eqref{costMertonNoConsCRRA1} have meaningful interpretations, as discussed in Appendices \ref{appendix effective derivation} and \ref{appendix competition of portfolio managers with Delayed Tax Effects}, respectively.

Lastly, for the CRRA case, the admissible set for $\pi^i$ is extended with additional requirements, i.e.,  we will take $\pi^i \in \mathbb{A}^i$:
\begin{equation}\label{CRRA no cons admissible set}
\mathbb{A}^i= \Big\{ \mathcal{F}_t^{\bm{X}}\mathrm{-progressively \ measurable \  } \pi^i : [0,T] \times \Omega \rightarrow \mathbb{R} \   \Big|  \ \exists K >0 : |\pi^i_t X^i_t| \leq K | Z^i_t| \  \Big\},
\end{equation}
where $\mathcal{F}_t^{\bm{X}}$ is the filtration generated by $(X^1_{(-\infty,t]}, \cdots, X^N_{(-\infty,t]})$. With $\pi^i \in \mathbb{A}^i$ and taking $\zeta^i(t)$  chosen such that $Z^i_t = X^i_t + aY^i_t > 0$ for all $t \leq 0$, we can show for all $i$, $Z^i_{disc,t} > 0$ a.s., and therefore the utility given by Eqs.~\eqref{costMertonNoConsCRRA1}--\eqref{costMertonNoConsCRRA2} is well defined. This is shown in Appendix \ref{CRRA utility proof no consump}.

The solution to the CARA case is summarized in the following proposition.

\begin{prop}\label{thmCARAnoconsump}
Consider the stochastic delay differential game defined by the dynamics \eqref{XdyanmicsMertonNoCons}  with reward for each player $i \in \{ 1, \cdots, N \}$ given by $J^i = J^i[\bm \pi]$ as defined through Eq.~\eqref{costMertonNoConsCARA1}--\eqref{a and Z_disc}. Then, there is a closed-loop Nash equilibrium $\bm \pi^*$ given by the controls
\begin{equation*}
        \pi^{i,*}_t X_t^{i, \ast}  =  \frac{\mu_1-r}{\sigma^2} \left( \delta_i + \frac{\theta_i\bar{\delta}}{1-\bar{\theta}} \right) \frac{1}{e^{-(r+\lambda a)t}},
\end{equation*}
where $\displaystyle\bar{\delta} = \frac{1}{N} \sum_{i=1}^N \delta_i$, $\displaystyle\bar{\theta} = \frac{1}{N} \sum_{i=1}^N \theta_i$, and $X_t^{i, \ast}$ satisfies the dynamics in \eqref{XdyanmicsMertonNoCons} associated with $\pi^{i, \ast}$. More precisely, the Nash equilibrium strategy at time $t$ is to invest a deterministic dollar amount into the risky asset, independent of the current wealth level.
\end{prop}
\begin{proof}
See Appendix \ref{CARA utility proof no consump}.
\end{proof}

\noindent This proposition characterizes the resulting Nash equilibrium for the CARA case. For the CRRA case, we have the following result.

\begin{prop}\label{thmCRRAnoconsump}
Consider the stochastic differential game  with delay defined by the dynamics \eqref{XdyanmicsMertonNoCons},  the reward $J^i = J^i[\bm \pi]$ for each player $i \in \{ 1, \cdots, N \}$ defined through Eq.~\eqref{a and Z_disc}--\eqref{costMertonNoConsCRRA2}, and the admissible space $\otimes_{i=1}^N \mathbb{A}^i$  given by Eq~\eqref{CRRA no cons admissible set}. Then, there is a closed-loop Nash equilibrium $\bm \pi^*$ given by the controls
\begin{equation*}
        \pi^{i,*}_t  =  \frac{\mu_1-r}{\sigma^2} \left( \delta_i - \frac{\theta_i(\delta_i-1)\bar{\delta}}{1+\olsi{\theta(\delta-1)}} \right) \frac{X^i_t + aY^i_t}{X^i_t},
\end{equation*}
where $\displaystyle\bar{\delta} = \frac{1}{N} \sum_{i=1}^N \delta_i$ and 
$\displaystyle\olsi{\theta(\delta-1)} = \frac{1}{N} \sum_{i=1}^N \theta_i(\delta_i-1)$.
\end{prop}
\begin{proof}
See Appendix \ref{CRRA utility proof no consump}.
\end{proof}

\subsection{Consumption and Portfolio Allocation Game with Delayed Tax Effects}\label{subSection the Competitive Consumption and Portfolio Allocation with Delayed Tax Effects problem}

In addition to the delay effects in investment strategies as analyzed in Section~\ref{subSection the competition of portfolio managers with Delayed Tax Effects}, here we also consider players' consumption strategies which contribute to their relative utility in the reward. Such a problem without delay was studied in \cite{lackerwconsump}.

As before, $\pi^i_t$ represents the fraction of player $i$'s wealth allocated to the risky asset at time $t$. The second control process, $c^i_t$, is person $i$'s rate of consumption at time $t$ as a fraction of her wealth. In addition to the usual admissibility conditions given by Eq.~\eqref{admissible_sets}, we require that $c^i_t \geq 0 $ for all $t \in (0,T]$. In this case, the wealth dynamics for player $i \in \{1,\cdots,N\}$ are given by
\begin{equation}\label{XdyanmicsMertonwCons}
\begin{aligned}
 & \ud X^i_t =  \left[(\mu_1 - r) \pi^i_t X^i_t  + rX^i_t -\mu_2 Y^i_t - c^i_t X^i_t  \right] \ud t + \sigma \pi^i_t X^i_t \ud W_t, \ t \in (0,T], \\
 & X^i_t = \zeta^i(t), \ t \in (-\infty,0], \\
\end{aligned}
\end{equation}
where $Y^i_t = \int_{-\infty}^t \lambda e^{-\lambda(t- s)} X^i_{s} \ud s$ is the exponentially decayed moving average of past wealth. The parameters are taken as $\mu_1,r,\lambda, \sigma \in \mathbb{R}$: $\lambda, \sigma>0$, $\mu_1 > r$, and $\mu_2 > 0$. The results will stay the same if one has $\mu_2 \in \mathbb{R}$, but only $\mu_2>0$ corresponds to taxes. We further require that $(r+\lambda)^2 - 4 \lambda \mu_2 > 0$ and $r+\lambda > 0$, as we did in Section~\ref{subSection the competition of portfolio managers with Delayed Tax Effects}. Again, the length of delay is $\tau = \infty$ as $Y^i_t$ in Eq.~\eqref{XdyanmicsMertonwCons} depends on the entire path $X^i_{(-\infty,t]}$.

With the dynamics fully described, we now define the reward function for player $i$ to be given by
\begin{equation}\label{costMertonwConsCRRA1}
  J^i[\bm \pi, \bm c] = \mathbb{E} \left[\int_0^T U^i(C_{disc,t}^i \olsi{C_{disc,t} }^{-\theta_i}) \ud t + \epsilon_i U_i \left( Z_{disc,T}^i \olsi{Z_{disc,T}}^{-\theta_i}  \right)  \right],
\end{equation}
where $0<\theta_i<1$, $\epsilon_i > 0$, and
\begin{equation}\label{costMertonwConsCRRA2}
    \begin{aligned}
   & Z^i_t = X^i_t + a Y^i_t, &  a = \frac{-(r+\lambda) + \sqrt{(r+\lambda)^2 - 4 \lambda \mu_2}}{2\lambda}, \\
   &  C_{disc,t}^i =  e^{-(r+\lambda a)t} c^i_t X^i_t, & \olsi{C_{disc,t}} =  \left(\prod_{i=1}^N C^i_{disc,t}  \right)^{1/N}, \\
    & Z_{disc,t}^i = e^{-(r+\lambda a)t}Z^i_t, 
    & \olsi{Z_{disc,t}} =  \left(\prod_{i=1}^N Z^i_{disc,t}  \right)^{1/N}, \\
    \end{aligned}
\end{equation}
and the utility function is the constant relative risk aversion (CRRA) utility given by
\begin{equation}\label{costMertonwConsCRRA3}
 U_i(z)  = \begin{cases}  
        \frac{1}{1- \frac{1}{\delta_i}} z^{1- \frac{1}{\delta_i}}, &  \delta_i \neq 1, \\
        \log(z), &  \delta_i = 1,\\
    \end{cases}\\
\end{equation}
with $\delta_i>0$. 

In this case, we can again interpret  $r+\lambda a$ as the ``tax-adjusted risk-free rate'' and $Z^i_t$ as ``tax-adjusted wealth'', and the interpretation of the expected utility \eqref{costMertonwConsCRRA1} is  discussed in Appendix \ref{appendix the Competitive Consumption and Portfolio Allocation with Delayed Tax Effects problem}. This example mainly contrasts with that in Section \ref{subSection the competition of portfolio managers with Delayed Tax Effects} in that each player now has two controls. This changes the mathematical structure and poses additional numerical challenges.

Lastly, to ensure the utility in Eqs.~\eqref{costMertonwConsCRRA1}--\eqref{costMertonwConsCRRA3} is well defined, we require that the controls for player $i$, $(\pi^i,c^i)$, live in the admissible set $\mathbb{A}^i$:
\begin{equation}\label{CRRA w cons admissible set}
\mathbb{A}^i= \Big\{ \mathcal{F}_t^{\bm{X}}\mathrm{-progressively \ measurable \  } (\pi^i,c^i) : [0,T] \times \Omega \rightarrow \mathbb{R} \cross \mathbb{R}^+  \Big|   \exists K >0 : |\pi^i_t X^i_t|, |c^i_t X^i_t| \leq K | Z^i_t| \  \Big\},
\end{equation}
where $\mathcal{F}_t^{\bm{X}}$ is the filtration generated by $(X^1_{(-\infty,t]}, \cdots, X^N_{(-\infty,t]})$. With $(\pi^i,c^i) \in \mathbb{A}^i$ and taking the initial path $\zeta^i$  chosen such that $Z^i_t = X^i_t + aY^i_t > 0$ for all $t \leq 0$, we can show for all $i$, $C^i_{disc,t}, Z^i_{disc,t} > 0$ a.s., see details in Appendix \ref{CRRA utility proof w consump}.

The characterization of the closed-loop Nash equilibrium is given by the following proposition.

\begin{prop}\label{thmCRRAconsump}
Consider the stochastic differential game  with delay defined by the dynamics \eqref{XdyanmicsMertonwCons}, the reward $J^i = J^i[\bm \pi, \bm c]$ for each player $i \in \{ 1, \cdots, N \}$ defined through Eqs.~\eqref{costMertonwConsCRRA1}--\eqref{costMertonwConsCRRA3}, and the admissible space $\otimes_{i=1}^N \mathbb{A}^i$ defined by Eq~\eqref{CRRA w cons admissible set}.
%\dbchange{and taken over the admissible space $\otimes_{i=1}^N \mathbb{A}^i$ defined by Eq~\eqref{CRRA w cons admissible set}}. 
Then,  there is a closed-loop Nash equilibrium $(\bm \pi^*, \bm c^*)$ given by the controls
\begin{equation}
  \begin{aligned}
        \pi^{i,*}_t & =  \frac{\mu_1-r}{\sigma^2} \left( \delta_i - \frac{\theta_i(\delta_i-1)\bar{\delta}}{1+\olsi{\theta(\delta-1)}} \right) \frac{X^i_t + aY^i_t}{X^i_t},\\
        c^{i,*}_t & = \begin{cases}
            \left(\beta_i^{-1} + (\gamma_i^{-1} - \beta_i^{-1}) e^{-\beta_i(T-t} \right)^{-1}\frac{X^i_t + aY^i_t}{X^i_t} , & \delta_i \neq 1, \\
            \left(T-t -\gamma_i^{-1} \right)^{-1} \frac{X^i_t + aY^i_t}{X^i_t}, & \delta_i = 1,\\   
        \end{cases} \\
    \end{aligned}  
\end{equation}
where $\displaystyle\bar{\delta} = \frac{1}{N} \sum_{i=1}^N \delta_i$, $\displaystyle\olsi{\theta(\delta-1)} = \frac{1}{N} \sum_{i=1}^N \theta_i(\delta_i-1)$ and  the parameters $\beta_i$ and $\gamma_i$ are given by
\begin{equation}
  \begin{aligned}
    \beta_i & = \frac{1}{2}(1-\delta_i) \left( \frac{\mu_1-r}{\sigma}\right)^2 \left(1- \frac{\theta_i \bar{\delta}}{1+\olsi{\theta(\delta-1)}} \right) \left( \delta_i - \frac{\theta_i \bar{\delta}}{1+\olsi{\theta(\delta-1)}} (\delta_i - 1)\right),\\
    \gamma_i & = \epsilon_i^{-\delta_i} \left( \left( \prod_{k=1}^N  \epsilon_k^{\delta_k} \right)^{1/N} \right)^{\theta_i(\delta_i-1)/(1+\olsi{\theta(\delta-1)})}.\\
    \end{aligned}  
\end{equation}
\end{prop}
\begin{proof}
See Appendix \ref{CRRA utility proof w consump}.
\end{proof}

\subsection{Inter-Bank Lending Model for Systemic Risk} \label{subSection Inter-Bank Lending Model}
The last problem we present comes from the study of systemic risk within inter-bank lending. The particular model we follow is a stochastic delay differential game introduced and studied by Carmona, Fouque, Mousavi, and Sun~\cite{CarmonaFouqueMousaviSun}. In this model, each bank lends/borrows monetary reserves to/from a central bank with their controls being their pace of lending/borrowing. The model includes loan repayments after a fixed time $\tau > 0$, which leads to the delayed component in the SDDE. The log-monetary reserves of bank $i$, $X^i_t$,  change in a differential manner with respect to this lending/borrowing dynamics along with some noise. Mathematically, $X_t^i$ satisfies the SDDE:
\begin{equation}\label{iblwd_eqn}
\begin{aligned}
  & \ud X^i_t = (\alpha^i_t - \alpha^i_{t-\tau}) \mathrm{d}t + \sigma \mathrm{d} W^i_t, &  \  t \in [0,T ],  \\
 & \alpha^i_{t} = 0 \in \mathbb{R}, & \ t \in [-\tau,0], \\
  & X^i_{0} = \xi^i \in \mathbb{R}. & \ \\ 
\end{aligned}
\end{equation}
The cost function for bank $i$ is given by
\begin{equation}\label{ibwld_loss}
J^i[\bm{\alpha}] =  \mathbb{E} \left[ \int_0^T  \left(\frac{1}{2}(\alpha^i_t)^2 - q \alpha^i_t (\bar{X}_t - X^i_t) + \frac{\epsilon}{2}(\bar{X}_t - X^i_t)^2  \right) \mathrm{d} t  + \frac{c}{2}(\bar{X}_T - X^i_T)^2    \right].
\end{equation}
%The variable $X^i_t \in \mathbb{R}$ represents the log-monetary reserves of bank $i$ at time $t$, and t
The control $\alpha^i_t \in \mathbb{R}$ is their corresponding pace of borrowing ($\alpha^i_t > 0$) or lending ($\alpha^i_t < 0$) at time $t$. Although the level of volatility, $\sigma > 0$, is the same for each bank, we have that $\{W^i\}_{i=1}^N$ are independent 1-D Brownian motions meaning that each bank experiences their own idiosyncratic noise, and $\bar X_t = \frac{1}{N}\sum_{i=1}^N X_t^i$.

A bank's choice of action is dictated by its incentive mechanisms illustrated through Eq.~\eqref{ibwld_loss}. The main incentive of bank $i$ can be simply stated as a desire to borrow when they deem their reserves to be ``too low'' and lend when deemed ``too high''. In particular, bank $i$ will arithmetically compare their level of log-monetary reserves to the mean log-monetary reserves of all banks, $\bar{X}$, with a preference for bank $i$ to have $X^i_t \approx \bar{X}_t$. This is exemplified by the terms $q \alpha^i_t (\bar{X}_t - X^i_t)$, $\frac{\epsilon}{2}(\bar{X}_t - X^i_t)^2$, and $\frac{c}{2}(\bar{X}_T - X^i_T)^2$ in Eq.~\eqref{ibwld_loss}. The parameters $q \geq 0$, $\epsilon > 0$ and $c \geq 0 $ respectively represent the degrees to which a bank desires 1) to borrow when there are too little monetary reserves 2) to maintain near average capitalization of log-monetary reserves at all times 3) to have near average capitalization at the final time. These incentives to have near average levels of log-monetary reserves are balanced by the bank's inclination to avoid lending or borrowing all else equal as represented by the quadratic penalty $\frac{1}{2}(\alpha^i_t)^2$ in Eq.~\eqref{ibwld_loss}.

%A bank's choice of action is dictated by their incentive mechanisms illustrated through Eq.~\eqref{ibwld_loss}. The main incentive of bank $i$ can be simply stated as a desire to borrow when they deem their reserves to be ``too low'' and lend when deemed ``too high''. In particular, bank $i$ will arithmetically compare their level of log-monetary reserves to the mean log-monetary reserves of all banks, $\bar{X}$, with a preference for bank $i$ to have $X^i_t \approx \bar{X}_t$. This is exemplified by the terms $q \alpha^i_t (\bar{X}_t - X^i_t)$, $\frac{\epsilon}{2}(\bar{X}_t - X^i_t)^2$, and $\frac{c}{2}(\bar{X}_T - X^i_T)^2$ in Eq.~\eqref{ibwld_loss}. The parameter $q \geq 0$  quantifies the desire to borrow when there are too little monetary reserves, $\epsilon > 0$ quantifies the desire for near average capitalization of log-monetary reserves at all times, and $c \geq 0 $ quantifies the desire for near average capitalization at the final time. These incentives to have near average levels of monetary reserves are balanced by the bank's inclination to avoid lending or borrowing all else equal as represented by the quadratic penalty $\frac{1}{2}(\alpha^i_t)^2$ in Eq.~\eqref{ibwld_loss}. \hdc{too many "desires" in the paragraph!} 

 The closed-loop Nash equilibrium for the stochastic delay differential game \eqref{iblwd_eqn}--\eqref{ibwld_loss} is derived and proven in \cite{CarmonaFouqueMousaviSun}. The result is restated in the proposition below for convenience.
 \begin{prop}[{\cite[Proposition 6.1]{CarmonaFouqueMousaviSun}}]\label{thm iblwd} Consider the stochastic differential game  with delay defined by the dynamics \eqref{iblwd_eqn} and with the reward for each player $i \in \{ 1, \cdots, N \}$ given by $J^i = J^i[\bm \alpha]$ as defined through Eq.~\eqref{ibwld_loss}. Then, there exists a closed-loop Nash equilibrium $\bm \alpha^*$ given by the control for each player $i \in \{1, \cdots, N \}$ by
\begin{equation*}
\begin{aligned}
\alpha^{i,*}_t   = 2 \left(1 - N^{-1} \right) \bigg[ & \left( E_1(t,0) + E_0(t) + \frac{q}{2(1-N^{-1})} \right) (\bar{X_t} - X_t^i)  \\
\ & + \int_{t-\tau}^t ( E_2(t,s-t,0) + E_1(t,s-t))(\olsi{\alpha^*}_{s} - \alpha^{i,*}_{s})\ud s  \bigg], \\
\end{aligned}
\end{equation*}
where $\displaystyle\olsi{\alpha^*} = \frac{1}{N}\sum_{i=1}^N \alpha^{i,*} $, and where $E_0, \cdots, E_2$ are given by the following PDE system in the region $(t,s,r) \in [0, T] \times [-\tau, 0] \times [-\tau, 0]$: %remark: E_3 equation and BC (commented out) can be included if interested in the value function as well. The value function contains E_3.
\begin{equation*}
\begin{aligned} 
%& E_3'(t)+\left(1-N^{-1}\right) \sigma^2 E_0(t)=0,\\ 
& E_0'(t)+\frac{\epsilon}{2} = 2\left(1-N^{-2}\right)\left(E_1(t, 0)+E_0(t)\right)^2+2 q\left(E_1(t, 0)+E_0(t)\right)+\frac{q^2}{2},  \\
&  \partial_t E_1(t, s) - \partial_s E_1(t, s) = 
2\left(1-N^{-2}\right)\left(E_1(t, 0)+E_0(t)+\frac{q}{2\left(1-N^{-2}\right)}\right)\left(E_2(t, s, 0)+E_1(t, s)\right),  \\
 & \partial_t E_2(t, s, r)- \partial_s E_2(t, s, r)-\partial_r E_2(t, s, r)  = 2\left(1-N^{-2}\right)\left(E_2(t, s, 0)+E_1(t, s)\right)\left(E_2(t, r, 0)+E_1(t, r)\right),
\end{aligned}
\end{equation*}
with boundary conditions given by
\begin{equation*}
\begin{aligned}
 & E_0(T)=\frac{c}{2}, \quad E_1(T, s)=0, \quad E_2(T, s, r)=0, \quad E_2(t, s, r)=E_2(t, r, s), \\
 & E_1(t,-\tau)=-E_0(t),  \quad E_2(t, s,-\tau)=-E_1(t, s). \\ %\quad E_3(T)=0.\\
\end{aligned}
\end{equation*}
\end{prop}
\noindent This Nash equilibrium is derived in \cite{CarmonaFouqueMousaviSun} by first formulating the delayed problem as a stochastic differential game in an infinite-dimensional Hilbert space, and then characterizing the Nash equilibrium through Hamilton-Jacobi-Bellman equations over a Hilbert space of functions. A thorough discussion of this technique as well as the theory of infinite-dimensional stochastic control problems can be found in \cite{SOCinfdim}.

\section{The Deep Learning Algorithm}\label{section the discrete problem and algorithm}
In Section \ref{The_Mathematical_Problem}, we presented the mathematical formulation for the Nash equilibrium problem of a stochastic delay differential game as the collection of conditions \eqref{general_SDDG_dynamics}--\eqref{general_NE}. We remarked that Eqs.~\eqref{general_SDDG_dynamics} and \eqref{general_SDDG_cost} allow one to define a map from choices of controls to the corresponding expected costs of each player. With respect to this map, one can define the notion of solution as that of Nash equilibrium~\eqref{general_NE}, where one searches for such equilibrium for ``allowed'' controls in an admissible set given by Eq.~\eqref{admissible_sets}. To approach this problem numerically, we will need to approximate the map from controls to cost functions by discretizing the dynamics of the game and expected cost. Additionally, to make the space of controls numerically tractable, we will have to represent the controls in a finite-dimensional space that still respects the closed-loop structure demanded by the problem. To this end, we describe in Section~\ref{subsection the discrete problem} the discretized set-up that allows us to define a map from a finite-dimensional space of RNN-based controls to the numerically approximated expected costs of each player. Such a discretization scheme is straightforward, but we decide to include it for the sake of clarity and completeness. Building on Section~\ref{subsection the discrete problem}, which focuses on approximating the cost of a single player given a choice of RNN-based controls, Section~\ref{subSection numerical algorithm} is dedicated to the algorithm for games with delay. 
%. This is a crucial component that is repeatedly utilized in the algorithm proposed in Section \ref{subSection numerical algorithm}. Specifically, by approximating the cost for a given selection of RNN-based controls, we can iteratively update the control choices using a combination of traditional machine learning optimization techniques such as stochastic gradient descent and game theory concepts such as fictitious play.
The algorithm we propose belongs to a broader methodology called deep fictitious play, introduced in \cite{DFP2,DFPSDG}. In Section \ref{subsubSection Deep Fictitious Play for approximating Nash equilibrium}, we provide a general discussion of the main ideas behind deep fictitious play and present the precise deep fictitious play algorithm we propose for solving stochastic delay differential games. Lastly, in Section \ref{Section lstm}, we highlight the specific choice of RNNs used in our implementation, namely the long short-term memory (LSTM) architecture.

\subsection{The Discrete Problem}\label{subsection the discrete problem}
We now consider the discrete analogue of the stochastic delay differential game defined by Eqs.~\eqref{general_SDDG_dynamics}--\eqref{general_SDDG_cost}. We represent the SDDE \eqref{general_SDDG_dynamics} numerically by its associated Euler Maruyama scheme, and the expected cost \eqref{general_SDDG_cost} is estimated with an empirical cost computed by the Monte Carlo method. Instead of searching for controls in the space given by Eq.~\eqref{admissible_sets}, we parameterize each player's control through a given RNN which helps to address the delayed aspect of the problem.

 The discrete analogue of the SDDE \eqref{general_SDDG_dynamics} is obtained using the Euler-Maruyama method, taking into account the delay. This is a natural choice for approximating SDDEs, and its convergence properties have been well-established in certain cases, for instance, Mao \cite{EulerMaruyamaAnalysisMao} addresses the case of a single pointwise delay. The discretization of the delay needs to be done on a case-by-case basis, and we will describe it later.
 
 For step size $\Delta t > 0$, we consider the partition of $[-\tau, T]$ given by $\{ t_k = k\Delta t :  -N_\tau \leq k \leq N_T, k \in \mathbb{Z} \}$, where $N_\tau = \frac{\tau}{\Delta t}$ and  $N_T = \frac{T}{\Delta t}$ are  integers without loss of generality\footnote{If the divisibility is not met, one may perturb $\Delta t$, $\tau$ and/or $T$ in order to ensure divisibility of $N_\tau = \frac{\tau}{\Delta t}, N_T = \frac{T}{\Delta t}$. Generality is still respected as the perturbations of each can be taken to be arbitrarily small.}. Then, we define the discrete approximation $(\hat{\bm{X}}_{k})_k$ through the Euler-Maruyama scheme
 \begin{equation}\label{general_numerical_SDDG_dynamics}
\begin{aligned}
     \hat{\bm{X}}_{k+1} & =   \ \hat{\bm{X}}_{k} 
    + \hat{\mu}(t_k, \hat{\bm{X}}_{k-N_{\tau}}, \cdots, \hat{\bm{X}}_{k},  \hat{\bm{\alpha}}_{{k-N_{\tau}}}, \cdots, \hat{\bm{\alpha}}_{k}) \Delta t  & \\
     &  +   \ \hat{\sigma}(t_k, \hat{\bm{X}}_{k-N_{\tau}}, \cdots, \hat{\bm{X}}_{k}, \hat{\bm{\alpha}}_{{k-N_{\tau}}}, \cdots, \hat{\bm{\alpha}}_{k}) \Delta \bm{W}_k,  \quad k = 0, \ldots, N_T-1, & \\
      \hat{\bm{X}}_{k} & =  \bm{\zeta}(t_k),  \quad  k = -N_{\tau}, \ldots, 0,  &  \\
     \hat{\bm{\alpha}}_{k} & =  \ \bm{\phi} (t_k), \quad k = -N_\tau, \ldots, -1,  &  \\
\end{aligned}
\end{equation}
where $\Delta \bm{W}_k = \bm{W}_{t_{k+1}} - \bm{W}_{t_k}$ from the original Brownian motion $\bm{W}$ in Eq.~\eqref{general_SDDG_dynamics}. The value for $\hat{\bm{\alpha}}_k = (\hat{\alpha}^1_k, \cdots, \hat{\alpha}^N_k)$ will be determined as the output of $N$ separate RNNs; each player's control is parameterized by their own RNN. %\rh{please rephrase this sentence} 

We have also approximated the functionals $\mu$ and $\sigma$ occurring in the SDDE \eqref{general_SDDG_dynamics} with discrete counterparts $\hat{\mu}$ and $\hat{\sigma}$. While the functionals $\mu, \sigma$ are generic, there will be natural choices for their discrete counterparts in some of the common cases that we consider. For example, the problems occurring in Sections \ref{subSection the competition of portfolio managers with Delayed Tax Effects} and \ref{subSection the Competitive Consumption and Portfolio Allocation with Delayed Tax Effects problem} contain a delay variable given by an integral over the past history of the state process. This can be approximated by a numerical quadrature along the partition or through discretization of a separate ODE that produces this integral. The problem introduced in Section~\ref{subSection Inter-Bank Lending Model} contains a delay variable in the form of pointwise evaluation of the control at a time $t-\tau$. This case is easily dealt with because the delay evaluation occurs on the partition as we have that $\tau$ and $T$ are both divisible by $\Delta t$.

Having defined the discrete approximation to the SDDE, the numerical approximation of the expected cost for player $i$ is given by
\begin{equation}\label{general_numerical_SDDG_cost}
    \hat{J}^i[\hat{\bm \alpha}] = \frac{1}{N_\mathrm{batch}} \sum_{\ell=1}^{N_\mathrm{batch}}  \left[ \sum_{k=1}^{N_T} f^i(t_k,\hat{\bm{X}}_k(\omega_\ell),\hat{ \bm  \alpha}_k)\Delta t + g^i(\hat{\bm{X}}_{N_T}(\omega_\ell))   \right],
\end{equation}
which is computed by taking $N_\mathrm{batch}$ samples of the discrete Brownian paths given by $\{ (\Delta \bm{W}_k(\omega_\ell))_{k = 0}^{N_T - 1} : \omega_\ell \in \Omega\}_{\ell = 1}^{N_\mathrm{batch}}$, producing the realized trajectories $(\bm{\hat{X}}_k(\omega_\ell))$ through the discrete dynamics \eqref{general_numerical_SDDG_dynamics}.

 We know that in the continuous version, each player's control lives in the admissible set $\mathbb{A}^i$ \eqref{admissible_sets}, which contains suitable $\mathcal{F}_t^X$-progressively measurable strategies, i.e, closed-loop controls. To capture this closed-loop aspect in the discrete case, we will require that $\hat{\bm\alpha}_k = (\hat{\alpha}^1_k, \cdots, \hat{\alpha}^N_k)$ be given as outputs of functions of the past history of the state space, $(\hat{\bm X}_{k'})_{k' \leq k}$. In particular, each player's strategy will be given through the outputs of an RNN of a fixed architecture. 
 
 The concept of RNN was first introduced by Rumelhart, Hinton, and Williams~\cite{rumelhart1986learning}, a work that demonstrates the implementation of the backpropagation algorithm of a neural network that includes hidden units. In our case, the RNN structure naturally allows the control for player $i$ to encapsulate the past history of the state process. We denote the RNN characterizing the actions of player $i$ by $\phi_{RNN}(\cdot ; \vartheta_i)$, where $\vartheta_i$ represents the parameters of the RNN for player $i$. Specifically, player $i$'s control at time $t_k$ for the discretized problem is given as a function of the input sequence $(t_{-N_\tau} , \hat{\bm X}_{-N_\tau} ), \cdots,  (t_{k} , \hat{\bm X}_{k})$, or
 \begin{equation}\label{alphaAsRNN}
     \hat{\alpha}^i_k = \phi_{RNN} \left( (t_{-N_\tau} , \hat{\bm X}_{-N_\tau} ), \cdots,  (t_{k} , \hat{\bm X}_{k}) ; \vartheta_i  \right).
 \end{equation} 
The map $\phi_{RNN}$ in Eq.~\eqref{alphaAsRNN} inputs a sequence of arbitrary length by defining it through a recurrence relation. Specifically, the recurrence is on a map we call the RNN cell, or $\phi_{RNNcell}$. The recurrence occurs through a secondary output of the RNN cell called the hidden state, which we label as $h^i$ for player $i$. In our case, the recurrence starts with an initial value for $h^i_{-N_{\tau}}$ by some specific choice $h_{init}$. Then, we define the recurrence relation
 \begin{equation}   
\begin{aligned}\label{RNNcell}
 y^i_k , h^i_k & =  \phi_{RNNcell}(  t_k, \hat{\bm X}_{k}, h^i_{k-1} ; \vartheta_i),
 \quad & k= - N_\tau + 1 , \cdots, N_T,  \\
  \hat{\alpha}^i_k  & =  y^i_k,
 \quad & k= 0 , \cdots, N_T - 1,  \\
 h^i_{-N_\tau} & = h_{init}. & \\
\end{aligned}
\end{equation}
The time-$k$ map of this recurrence relation defines a map from $ \left((t_{-N_\tau} , \hat{\bm X}_{-N_\tau} ), \cdots,  (t_{k} , \hat{\bm X}_{k})  \right) \mapsto y^i_k = \hat{\alpha}^i_k$, which is precisely the map we call $\phi_{RNN}$ in Eq.~\eqref{alphaAsRNN}. We remark that the hidden states $h^i_k$ will correspond to each player $i$ as they are dependent on the parameters $\vartheta_i$. 

The key observation is that taking the control to be given by the RNN defined by Eqs.~\eqref{alphaAsRNN}--\eqref{RNNcell} provides us with a reasonable space to approximate the closed-loop controls in Eq.~\eqref{admissible_sets}. For one, we see that $\hat{\alpha}^i_k$, the discrete control output at time $t_k$, now depends on the past trajectory of $\hat{\bm X}$ up to and including time $t_k$, which encapsulates the closed-loop property. This also addresses the impact of the delay as the control at time $t_k$ has memory of past events. At the same time, the dimension of the search space of the control for each player $i$ is reduced to the finite dimension $\dim (\vartheta_i) < \infty$, which allows for tractability of the Nash equilibrium problem as we will see in Section \ref{subsubSection Deep Fictitious Play for approximating Nash equilibrium}. 

In essence, each player $i$ will select their desired neural network parameters $\vartheta_i$, determining their choice of control. With each player's control chosen, the recurrence relations given by both Eq.~\eqref{general_numerical_SDDG_dynamics}, \eqref{RNNcell} are then iterated together, which produces simulated dynamics for $\hat{\bm X}$. With these simulated dynamics, one can compute the empirical costs for each player $(\hat{J}^i)_{i=1}^N$. This defines a map from ``controls'' given by the choices of parameters $(\vartheta_i)_{i=1}^N$ to the empirical costs $(\hat{J}^i)_{i=1}^N$, which will allow us to proceed in Section \ref{subsubSection Deep Fictitious Play for approximating Nash equilibrium}
 with a deep fictitious play algorithm for approximating the Nash equilibrium.

\subsection{The Numerical Algorithm} \label{subSection numerical algorithm}

\subsubsection{Deep Fictitious Play for Approximating Nash Equilibrium} \label{subsubSection Deep Fictitious Play for approximating Nash equilibrium}

Deep fictitious play is a broad technique introduced in \cite{DFP2,DFPSDG} whereby Nash equilibrium controls are approximated iteratively via deep learning techniques in a manner akin to Brown's fictitious play~\cite{FP}. The primary reason for using deep learning is the high dimensionality in the problem being solved in each iterative step of fictitious play when one has a stochastic differential game.

In general, we consider a game defined by a map from choices of controls $\bm \alpha \in \mathbb{A}$ into costs $\bm J[\bm \alpha]$. The idea of fictitious play is to fix all but one player's control, which leads to the decoupled optimization problems
\begin{equation}\label{fp_opt_eq}
    \inf_{\beta^i \in \mathbb{A}^i} J^i[\alpha^{1}, \cdots, \alpha^{i-1}, \beta^i, \alpha^{i+1}, \cdots, \alpha^{N}],
\end{equation}
and then we iterate over the solutions to these optimization problems. 

Assuming a unique minimum occurs at $\beta^{i,*}$, we use $\beta^{i,*}$ to inform $\alpha^{i}$ in future iterations of the optimization \eqref{fp_opt_eq}. Doing this for each player $i \in \{1, \cdots, N\}$ constitutes one round of this modified fictitious play. In the case of Brown's fictitious play, $\alpha^{j}$ in the optimization problem~\eqref{fp_opt_eq} would be given by the empirical average of player $j$'s control taken over the previous rounds of play. However, for deep fictitious play methodologies, we will usually take $\alpha^{j}$ to be the exact control from the previous round of play for memory efficiency reasons (see \cite[Remarks~3.1 and 3.2]{DFP2} for more information).

In this approach, we choose $N_{\mathrm{stages}}$ to be the number of stages of fictitious play. Player $i$ at stage $s$ selects her best response given that all other players are using their strategies from the previous round. This leads to the theoretical Algorithm \ref{math_alg} below. 

\begin{algorithm}[H]
\caption{Modified Fictitious Play}\label{math_alg}
\begin{algorithmic}[1]
\State{Initialize each $\alpha^{i,0} \in \mathbb{A}^i$}.
\For{$s$ in $1$ to $N_{\mathrm{stages}}$}
    \For{$i$ in $1$ to $N$}
        \State{$\alpha^{i,s} = \arg \min_{\beta^i \in \mathbb{A}^i} J^i[\alpha^{1,s-1}, \cdots, \alpha^{i-1,s-1}, \beta^i, \alpha^{i+1,s-1}, \cdots, \alpha^{N,s-1}]$ }
    \EndFor
\EndFor
\end{algorithmic}
\end{algorithm}
 In Algorithm \ref{math_alg}, $\alpha^{i,s}$ is the control for player $i$ at stage $s$, and we are assuming that the  minimizer exists and is unique. If it is not unique, we could choose a particular minimizer. The idea of  Algorithm \ref{math_alg} is that the Nash equilibria are  characterized precisely by the fixed points of this iteration. However, the convergence of the method outlined by Algorithm \ref{math_alg} is done on a case-by-case basis. For example, in \cite{DFPSDG} it is shown that Algorithm \ref{math_alg} converges for a Linear-Quadratic stochastic differential game. 

Next, Algorithm \ref{math_alg} is purely theoretical as it assumes the solution to the optimization problem \eqref{fp_opt_eq}. In reality, \eqref{fp_opt_eq} may be difficult to directly solve due to high dimensionality and may be best approached by deep learning techniques. There are several methods we may take to solve \eqref{fp_opt_eq} via deep learning and any of these would be considered deep fictitious play.  Namely, the approach in \cite{DFP2} considers the Hamilton-Jacobi-Bellman (HJB) equation given by the decoupled optimal control problem defined by the optimization problem~\eqref{fp_opt_eq}. The HJB solution can be framed in terms of a system of backward SDEs (BSDEs) which can be solved via the Deep BSDE method introduced by E, Han, and Jentzen \cite{deepBSDE1, deepBSDE2}. Alternatively, the solution of the HJB equation could be approximated with the Deep Galerkin method introduced by Sirignano and Spiliopoulos \cite{DGM1}.

In the case of stochastic delay differential games, the associated HJB equation would be infinite-dimensional \cite{SOCinfdim}. Because of this, we use a third approach to solve the optimization problem~\eqref{fp_opt_eq} that is based on the so-called direct parametrization as discussed by Han and E in \cite{han2016deep}. This approach was originally introduced for stochastic control problems but can be extended into the game setting via the iteration in Algorithm \ref{math_alg} as demonstrated in \cite{DFPSDG}. In this method, we approach the optimization problems in Algorithm \ref{math_alg} with neural networks. This is done by taking the control $\beta^i$ to be a neural network and the optimization is done with gradient descent using the cost function $J^i$ as the loss.

In our case, there are a few steps that must be computationally approximated. First, we must approximate the true expected cost $J^i$ with its empirical, discrete counterpart $\hat{J}^i$, Eq.~\eqref{general_numerical_SDDG_cost}. This is computed under controls given by the RNNs, $(\phi_{RNN}(\cdot, \vartheta_1), \cdots, \phi_{RNN}(\cdot, \vartheta_N))$, which give the choices of controls for each player defined through Eqs.~\eqref{alphaAsRNN}--\eqref{RNNcell} as discussed in Section \ref{subsection the discrete problem}. This leads us to Algorithm \ref{Deep fictitious play alg direct param} which is the basis of our proposed numerical method.

\begin{algorithm}[H]
    \caption{A Deep Fictitious Play Algorithm via Direct Parametrization}\label{Deep fictitious play alg direct param}
    \begin{algorithmic}[1]
    \State{Initialize each $\vartheta_{1,0}, \cdots, \vartheta_{N,0}$ which are the respective parameters of the $N$ different RNNs at stage 0.} 
    \State{Select $N_{\mathrm{stages}}$ of deep fictitious play based on the computational budget.}
    \For{$s$ in $0$ to $N_{\mathrm{stages}}$}
        \For{$i$ in $1$ to $N$}
            \State{Compute $N_\mathrm{batch}$ trajectories $\bm{\hat{X}}$ of the numerical SDDE \eqref{general_numerical_SDDG_dynamics} under the given controls $(\hat{\alpha}^{j,s})_j$,
            
            \quad where the controls $\hat{\alpha}^{j,s} = \phi_{RNN}(\cdot ; \vartheta_{j,s})$ are defined by Eqs.~\eqref{alphaAsRNN}--\eqref{RNNcell} for each player $j$.}
            \State{Compute the numerical cost $\hat{J}^i$  from Eq.~\eqref{general_numerical_SDDG_cost}}.
            \State{Compute via automatic differentiation $\nabla_{\vartheta_{i,s}} \hat{J}^i$ .}
            \State{Do a gradient descent step or similar (e.g. Adam)  on $\vartheta_{i,s}$ with learning rate $l_r$, i.e.
            
            \quad $\vartheta_{i,s+1} = \vartheta_{i,s} - l_r \nabla_{\vartheta_{i,s}} \hat{J}^i$. }
            
        \EndFor
    \EndFor
    \end{algorithmic}
\end{algorithm}

\noindent As before, we consider $N_{\mathrm{stages}}$ of iteration, where $\phi_{RNN}(\cdot, \vartheta_{j,s})$ parameterizes the control for player $j$ at stage $s$ as selected by the neural network parameters $\vartheta_{j,s}$. As motivated by Algorithm \ref{math_alg}, we would then like to compute the optimal $\vartheta_{j,s}$ holding the other RNN parameters, $(\vartheta_{j',s})_{j' \neq j}$, fixed.  In practice,  we  do  a gradient descent step or a sequence of gradient descent steps to approximate this behavior. Of course, to do this, we will have to compute the $\vartheta_{i,s}$-gradient of $\hat{J}^i$. This is possible numerically through automatic differentiation \cite{automatic_diff_reference}. Also, instead of standard gradient descent, we choose to use the Adam optimization which adaptively chooses the learning rate based on the mean and variance of the gradients involved in the computation of the loss. The choice of using the Adam optimization over traditional gradient descent is due to it having improved convergence properties in many cases \cite{adam_reference}. This gives us the deep fictitious play algorithm  shown above in Algorithm \ref{Deep fictitious play alg direct param}.

To better approximate the argument minimizer in Algorithm \ref{math_alg}, several gradient steps might be needed. However, this would require a new computation of $\hat{J}^i$ and its gradient with respect to the updated RNN parameters after each gradient descent step, leading to increased costs.\footnote{One possibility would be to incorporate having additional gradient descent steps for a single player before moving onto the next when one is at later stages of play. The idea is that at the beginning stages, since the approximate controls $(\hat{\alpha}^{i,s})_i$ are not yet close to the Nash equilibrium, it is not as important to fully optimize a single player given the choices of others since the other players are not yet close enough to the Nash equilibrium. However, at later stages, as the Nash equilibrium is approached, it may be beneficial to more fully optimize $\hat{\alpha}^{i,s}$ given the choices of the other players. } Instead, in Algorithm \ref{Deep fictitious play alg direct param}, we move on immediately to the next player's optimization after a single gradient descent step of the current player. For this reason, Algorithm \ref{Deep fictitious play alg direct param} is not meant to be a perfect numerical analogue of Algorithm \ref{math_alg}, but is instead merely based on it. Importantly,  Algorithm \ref{Deep fictitious play alg direct param} still reflects the property that Nash equilibria are fixed points of the iteration from an intuitive point of view.\footnote{Of course it is exceedingly unlikely that controls of the form $(\phi_{RNN}(\cdot,\vartheta^{i,*}))_i$ happen to be a Nash equilibrium due to it being a finite-dimensional object in an infinite-dimensional space. However, if it so happens to be the case that $(\phi_{RNN}(\cdot,\vartheta^{i,*}))_i$ is a Nash equilibrium for $\hat{\bm{J}}$, and assuming sufficient smoothness of $\hat{\bm J}$, then for each $i$ it holds that $\nabla_{\vartheta^{i,*}} \hat{J}^i [(\phi_{RNN}(\cdot,\vartheta^{j,*}))_j] = 0$, and therefore $(\vartheta^{i,*})_{i=1}^N$ is a fixed point of the iteration. }

\subsubsection{The Long Short-Term Memory Recurrent Network}\label{Section lstm}
 Algorithm \ref{Deep fictitious play alg direct param} showcases the numerical method we use to approximate the Nash equilibrium controls of stochastic delay differential games using RNN-based controls. It only remains to specify a precise form of the RNNs as we have presented these networks quite broadly through Eqs.~\eqref{alphaAsRNN}--\eqref{RNNcell}. Motivated by the implementation in \cite{RNNSCPwD}, we choose to use the specific RNN architecture given by the so-called long short-term memory (LSTM).

The LSTM was first introduced by Hochreiter and Schmidhuber \cite{LSTM} and was built to effectively handle the vanishing gradient problem. The LSTM will have two variables that both play the role of the hidden state of an RNN demonstrated in Eq.~\eqref{RNNcell}. Confusingly, one is called the hidden state $h$ and the other is the cell state $c$, although we will see they are both defined recurrently and therefore act as hidden states with respect to the generic RNN architecture we have defined by Eq.~\eqref{RNNcell}. The map $\phi_{LSTM}$ maps an input vector $(x_0, \cdots, x_k)$ of arbitrary length to an output, hidden, and cell state through a recursive dependence on its previous outputs of the previous input $(x_0, \cdots, x_{k-1})$. The recurrence is given through a function called the LSTM cell, which we will denote $\phi_{LSTMcell}$. In particular $\phi_{LSTMcell}$ directly maps the inputs $(x_k, c_{k-1}, h_{k-1})$ to the outputs $c_k, h_k, y_k$ according to
\begin{equation}\label{lstmcell} 
\begin{aligned}
 & i_k =  \sigma(W_i x_k + U_i h_{k-1} + b_i),  \\
 & f_k = \sigma(W_f x_k + U_f h_{k-1} + b_f) , \\
 & o_k =  \sigma(W_o x_k + U_o h_{k-1} + b_o) ,\\
  & c_k = f_k \odot c_{k-1} + i_k \odot \tanh(W_c x_k + U_c h_{k-1} + b_c), \\
  & h_k = o_k \odot  \tanh(c_k), \\
  & y_k = W_y h_k + b_y.\\
\end{aligned}
\end{equation}
 The individual mappings within the LSTM cell to $i_k, f_k,$ and $o_k$ are known as the input, forget, and output gate respectively. Denoting the input dimension, $\dim(x_k)$, to be $N_{input}$, and denoting the hidden dimension $N_{hidden}$ for the size of each $i_k, f_k, \cdots, h_k$, we see that each matrix $W_i, W_f, \cdots, W_c$ is of size $N_{hidden} \times N_{input}$ and the bias vectors $b_i, b_f, \cdots, b_c$ are of size $N_{hidden}$. The final output is $y_k \in \mathbb{R}^{N_{output}}$, so $W_y$ is in $\mathbb{R}^{N_{output} \times N_{hidden}}$ and $b_y$ is in $\mathbb{R}^{N_{output}}$. For player $j$, the LSTM cell map, $\phi_{LSTMcell}(\cdot;\vartheta_j)$, is determined by the choice of parameters $\vartheta_j = (W^j_i, \cdots, W^j_y, U^j_i, \cdots, U^j_c, b^j_i, \cdots, b^j_y)$, representing the weight matrices and bias vectors in Eq.~\eqref{lstmcell} specifically for player $j$.

With the cell-map, $\phi_{LSTMcell}$, specified by Eq.~\eqref{lstmcell}, the choice of controls in Eq.~\eqref{general_numerical_SDDG_dynamics} is determined by taking player $j$'s control at time $t_k$ to be
\begin{equation*}
    \hat{\alpha}^j_k = \phi_{LSTM}\left((-\tau, \hat{\bm X}_{-N_\tau}), \cdots,  (t_k,\hat{\bm X}_{k}) ; \vartheta_j \right),
\end{equation*}
where this mapping $\phi_{LSTM}(\cdot ; \vartheta_j)$ is given by the forward iteration of the recurrence relation
\begin{equation}\label{lstm_eq}
\begin{aligned}
  x_k & = (t_k, \hat{\bm{X}}_k), \quad & k= - N_\tau , \cdots, N_T,\\
  y^j_k, c^j_k, h^j_k & =  \phi_{LSTMcell}(x_k, h^j_{k-1}, c^j_{k-1} ; \vartheta_j), \quad & k= - N_\tau +1 , \cdots, N_T-1,  \\
    \hat{\alpha}^j_k & =  y^j_k, \quad & k = 0 , \cdots, N_T-1. \\
\end{aligned}
\end{equation} 
 In the cases where the dimension of the state process is equal to the number of players (i.e., $n=N$), we will choose $h^j_{-N_\tau } = c^j_{-N_\tau} = (\hat{X}^j_0, 0, \cdots, 0) \in \mathbb{R}^{N_{hidden}}$ to start the forward iteration, following the implementation in \cite{RNNSCPwD}. Note that $x_k = (t_k, \hat{\bm{X}}_k)$ is a vector in $\mathbb{R}^{1+n}$, as $\hat{\bm{X}}_k$ is a vector in $\mathbb{R}^n$ for each $k$. This means that while the input dimension is fixed $N_{input} = 1+n$, one is free to choose the size of the hidden dimension, $N_{hidden}$, depending on the user's desired size of the network. In our implementation, we choose $N_{hidden} = 2^6$.

\subsubsection{Implementation Details}
For computational efficiency, we may not be inputting a single sample of $(t_k, \hat{\bm{X}}_k)$ into the LSTM as indicated by Eq.~\eqref{lstm_eq}, but rather a so-called ``batch'' which contains $N_\mathrm{batch}$ paths of $\hat{\bm{X}}$ determined by respective $N_\mathrm{batch}$ samples of Brownian paths. 

Precisely, we can augment the operations in the Euler-Maruyama method \eqref{general_numerical_SDDG_dynamics} so that it iterates over a batch $(\hat{\bm{X}}_k(\omega_\ell))_{\ell=1}^{N_\mathrm{batch}}$ which is represented as a matrix in $\mathbb{R}^{n \times N_\mathrm{batch} }$. This is done by generating $N_\mathrm{batch}$ samples of the Brownian increments $({\Delta \bm{W}}_k(\omega_\ell))_{\ell=1}^{N_\mathrm{batch}}$. The drift $\hat{\mu}$ and volatility $\hat{\sigma}$ are extended to act pointwise across the batch dimension.

The control given by the neural network must also be able to provide the respective outputs for each sample of $\hat{\bm{X}}$ along the batch dimension by acting pointwise across the batch dimension. Note that a generic linear layer $x \mapsto Wx + b$ extends to the mapping  $(x_1, \cdots, x_{N_\mathrm{batch}})  \mapsto W(x_1, \cdots, x_{N_\mathrm{batch}}) + (b, \cdots, b)$, with the property that $x_i \mapsto W x_i + b$. Because of this, we see that the map $\phi_{LSTMcell}$ in Eq.~\eqref{lstmcell} naturally acts pointwise along the batch dimension. To have the LSTM defined by Eqs.~\eqref{lstmcell}--\eqref{lstm_eq} extended to act pointwise on the batch, we will take the input vector $x_k$ to be
\begin{equation*}
    x_k = (t_k, \hat{\bm{X}}_k(\omega_\ell))_{\ell=1}^{N_\mathrm{batch}} \in \mathbb{R}^{(1 + n) \times N_\mathrm{batch} }.
\end{equation*}
We notice that this will imply that $h^j_{k},c^j_{k}$ in Eq.~\eqref{lstm_eq} must also be tensorized along this batch dimension and we will have $c_k, h_k \in \mathbb{R}^{N_{hidden} \times N_\mathrm{batch}}$. For the operations in Eq.~\eqref{lstmcell} to act on a batch, we will have to resize the bias vectors $b_i, \cdots, b_c \in \mathbb{R}^{N_{hidden}}$ to be of size $\mathbb{R}^{N_{hidden} \times N_\mathrm{batch}}$ by repeating their original values across the batch dimension. The matrices $W_i, \cdots, W_c$ will remain the same size of $\mathbb{R}^{N_{hidden} \times (1+n)}$. %Edit1: With these extensions, we see that the map defined in Eq.~\eqref{lstmcell} acts pointwise across the batch dimension for the input $x_k = (t_k, \hat{\bm{X}}_k(\omega_\ell))_{\ell=1}^{N_\mathrm{batch}}$. %Original:With these extensions, we see that the maps in Eq.~\eqref{lstmcell}, acting on the input $x_k = (t_k, \hat{\bm{X}}_k(\omega_\ell))_{\ell=1}^{N_\mathrm{batch}}$, are equivalent to concatenating the $N_\mathrm{batch}$ copies of applying \eqref{lstmcell} pointwise to each $(t_k, \hat{\bm{X}}_k(\omega_\ell))$. \hdc{equivalent to concatenating?}

The end result is the ability to work with the map from the choice of controls $(\phi^1_{LSTM}(\cdot, \vartheta_1), \cdots, \phi^1_{LSTM}(\cdot, \vartheta_N))$ to the cost function $\hat{J}^i$ in Eq.~\eqref{general_numerical_SDDG_cost} in a tensorized form.
 From a computational perspective, we avoid looping over each sampled trajectory to compute the numerical cost function \eqref{general_numerical_SDDG_cost}. This tensorization is especially useful when working with automatic differentiation supported libraries such as PyTorch or TensorFlow, as these tensorized operations can be easily and automatically parallelized \cite{PyTorchRef,TFRef}.

\section{Numerical Results}\label{Section Numerical results}

In Section \ref{Section benchmark examples}, we have presented three  stochastic delay differential games with analytical formulas of their closed-loop Nash equilibrium. In each case, we now numerically approximate the closed-loop Nash equilibrium for $N=10$ players via our proposed numerical method,  Algorithm \ref{Deep fictitious play alg direct param} in Section \ref{subSection numerical algorithm}.

We introduce below in Section \ref{Numerical Results Methodology} the precise details of our numerical experiments that serve as a reference point for the construction of the plots shown in Sections \ref{subSection numerical Merton no cons}--\ref{subSection numerical inter-bank}, the parameter values used for each of these problems, and an important implementation detail for the problems with infinite delay. In Sections \ref{subSection numerical Merton no cons}--\ref{subSection numerical inter-bank} we present and interpret the numerical results for each of the considered problems.

\subsection{Numerical Results Methodology}\label{Numerical Results Methodology}

\subsubsection{Costs/Rewards over Training}\label{cost graph method}
For typical machine learning problems, one usually has a training curve-- a plot of the loss function over the course of training, which serves as an initial gauge of the effectiveness of training. While the loss function for each player can be seen through the player's empirical cost \eqref{general_numerical_SDDG_cost}, the controls are meant to approximate a Nash equilibrium;  the trajectory of each player's loss function over the course of training is not an appropriate measure of the effectiveness of training in this case. The impact of training can be better seen through the relative error of these costs under the LSTM controls to that corresponding to the true Nash equilibrium controls. 

Therefore, for every 20 rounds of deep fictitious play (DFP) within Algorithm \ref{Deep fictitious play alg direct param}, we compute the 2-norm relative error:
\begin{equation}\label{2 norm relative error}
    \mathrm{ Relative \ 2 \text{-} Norm \ Error} = \frac{|| \hat{\bm J}[{\bm \phi}_{LSTM}] -  \hat{\bm J}[\bm{\alpha}^*] ||_2}{||  \hat{\bm J}[\bm{\alpha}^*] ||_2},
\end{equation}
 where $\hat{\bm J}[{\bm \phi}_{LSTM}]$ and $\hat{\bm J}[\bm{\alpha}^*]$ respectively are the vectors containing the empirical cost for each player under the LSTM controls $(\phi^1_{LSTM}, \cdots, \phi^N_{LSTM})$ defined in Eqs.~\eqref{lstmcell}-- \eqref{lstm_eq} and the true Nash equilibrium controls $(\alpha^{*1} , \cdots, \alpha^{*N})$ of the mathematical problem defined by Eqs.~\eqref{general_SDDG_dynamics},\eqref{general_SDDG_cost}. We then plot this relative 2-norm error as it evolves over the course of training for each of the problems we consider throughout Sections \ref{subSection numerical Merton no cons}--\ref{subSection numerical inter-bank}.

\subsubsection{Comparison of State and Control Trajectories}\label{trajectory graph method}

After training, we have the collection of LSTM control functions for each player, $(\phi_{LSTM}(\cdot;\vartheta_i))_{i=1}^N$. We demonstrate the ability of these surrogate functions to approximate the true Nash equilibrium controls by comparing the trajectories of both control and state processes under both the LSTM and true Nash equilibrium controls for a given sample of Brownian motion.

This is done by selecting a single realization of the discrete Brownian motion's path, $(\Delta \bm{W}_k(\omega)_{k = 0}^{N_T - 1}$, and with this given noise simulate the discretized dynamics \eqref{general_numerical_SDDG_dynamics} once under the Nash equilibrium controls and again under the LSTM controls. We then compare the dynamics given the two different choices of controls. The plots of these dynamics are shown for each problem occurring throughout Sections \ref{subSection numerical Merton no cons}--\ref{subSection numerical inter-bank}. Each player is distinguished with a given color, while solid and  dashed lines correspond to the dynamics under LSTM controls and Nash equilibrium controls respectively. Lastly, while we have performed the training for 10 players to demonstrate the methodologies' ability to handle larger games, we will only plot 5 out of 10 players’ trajectories for the sake of visual clarity.

\subsubsection{Model Parameters}

The model parameters chosen for the numerical experimentation of each problem are shown below. We express a dependency on $i$ for parameters specific to player $i$, e.g. $\delta_i = 0.3 + \frac{4}{9}(i-1)$ in Table \ref{table param CARA no consump} is player $i$'s risk tolerance parameter. Here, we are still using the indexation $i \in \{1, \cdots, N = 10\}$.
\begin{table}[H] 
\caption{Parameters for CARA case of competition between portfolio managers with delayed tax effects.}\label{table param CARA no consump}
\centering
\begin{tabular}{| c || c | c | c | c | c | c | c | c | c | c | }
\hline
Parameter & $N$ & $T$ & $\mu_1$ & $\sigma$ & $r$ & $\lambda$ & $\mu_2$ & $\delta_i$ & $\theta_i$  & $X^i_{(-\infty,0]} = x^i_0$ \\
\hline
Value & 10 & 10.0 & 0.08 & 0.2 & 0.04 & 2.0 & 0.01 & $0.3 + \frac{4}{9}(i-1)$ & $0.3 + \frac{4}{9}(i-1)$ &  $2 + \frac{1}{10}(i-1)$    \\
\hline
\end{tabular}
\end{table}

%\noindent Note that we are assuming the initial path of wealth for each player is constant.

\begin{table}[H]
\caption{Parameters for CRRA case of competition between portfolio managers with delayed tax effects.}\label{table param CRRA no consump}
\centering
\begin{tabular}{| c || c | c | c | c | c | c | c | c | c | c | }
\hline
Parameter & $N$ & $T$ & $\mu_1$ & $\sigma$ & $r$ & $\lambda$ & $\mu_2$ & $\delta_i$ & $\theta_i$  & $X^i_{(-\infty,0]} = x^i_0$ \\
\hline
Value & 10 & 1.0 & 0.08 & 0.2 & 0.04 & 1.0 & 0.2 & $0.3 + \frac{4}{9}(i-1)$ & $0.3 + \frac{4}{9}(i-1)$ &  $1 + \frac{1}{20}(i-1)$    \\
\hline
\end{tabular}
\end{table}

%\noindent We remark that we have increased the size of $\mu_2$ substantially, and it is relatively large compared to what we should expect in a real-world scenario. The reason being is that the form of the Nash equilibrium control $\pi^{i,*}_t$ is proportional to $(X^i_t + aY^i_t)/X^i_t$  (see Proposition \ref{thmCRRAnoconsump}). Therefore, if $aY^i_t \ll X^i_t$ the control will be approximately constant. By increasing $\mu_2$ substantially, we increase $a$ providing us with a more interesting problem to solve.

\begin{table}[H]
\caption{Parameters for consumption and portfolio allocation game with delayed tax effects.} \label{table param CRRA w consump}
\centering
\begin{tabular}{| c || c | c | c | c | c | c | c | c | c | c | c | }
\hline
Parameter & $N$ & $T$ & $\mu_1$ & $\sigma$ & $r$ & $\lambda$ & $\mu_2$ & $\delta_i$ & $\theta_i$  & $X^i_{(-\infty,0]} = x^i_0$ & $\epsilon_i$ \\
\hline
Value & 10 & 2.0 & 0.08 & 0.2 & 0.04 & 1.0 & 0.01 & $0.3 + \frac{4}{9}(i-1)$ & $0.3 + \frac{4}{9}(i-1)$ &  $1 + \frac{1}{20}(i-1)$ & 50.0   \\
\hline
\end{tabular}
\end{table}
%\noindent Here, we have added the new parameter $\epsilon_i$ which weights the importance of terminal wealth for player $i$. The specific choice of $\epsilon_i$ has been chosen to enhance model realism chosen to make sure the utility of consumption does not dwarf the utility of terminal wealth and vice-versa.
\begin{table}[H]
\caption{Parameters for inter-bank lending model.} \label{table iblwd params}
\centering
\begin{tabular}{| c || c | c | c | c | c | c | c | c | }
\hline
Parameter & $N$ & $T$ & $\sigma$ & $q$ & $\epsilon$ & $c$ & $\tau$  & $X^i_0 = \xi^i$ \\
\hline
Value & 10 & 1.0 & .05 & 1.0 & 2.0 & 0.25 & 0.25 & $1+0.1 \cdot 1.15^{i-1}$\\
\hline
\end{tabular}
\end{table}
%\noindent Here, the length of delay, $\tau$, is a configurable model parameter as compared to the other cases where $\tau = \infty.$
\noindent Note that in Tables \ref{table param CARA no consump}--\ref{table param CRRA w consump}, we write $X^i_{(-\infty,0]} = x^i_0$, indicating the initial path for each player is taken to be constant.

\subsubsection{Approaching the Infinite Delay Cases}

We remark on an approximation used in the infinite delay cases appearing in the problems in Sections \ref{subSection the competition of portfolio managers with Delayed Tax Effects} and \ref{subSection the Competitive Consumption and Portfolio Allocation with Delayed Tax Effects problem} whose numerical results we display in  \ref{subSection numerical Merton no cons} and \ref{subSection numerical Merton with cons}. In both cases, the delay is contained through the variable $Y^i_t = \int_{-\infty}^t \lambda e^{-\lambda(t- s)} X^i_{s} \ud s$. 

While one option is to truncate the delay resulting in a truncated integral for which a standard numerical integration approach can be applied, this is not necessary. The reason is that in each of these problems, one can show $Y^i_t$ satisfies the ODE relation $\ud Y^i_t = \lambda (X^i_t - Y^i_t) \ud t$, allowing us to approximate $\bm Y_t$ through the forward Euler discretization. Moreover, when the initial path $X^i_{(-\infty,0]}$ is constant, we can easily see that $Y^i_0 = X^i_0$, which greatly simplifies the discrete SDDE iteration by altogether avoiding integration, and therefore the truncation of $\tau = \infty$ is no longer important for $\bm Y_t$. However, we still must impose some finite truncation to the delay $\tau$ in order for the forward iteration of the LSTM as described in Eq.~\eqref{lstm_eq} to be initialized at some finite $-N_\tau$. For the numerical results shown in both Sections \ref{subSection numerical Merton no cons} and \ref{subSection numerical Merton with cons}, we have used $\tau = 1.0$ as the truncation for the infinite delay.

\subsection{Results for Competition between Portfolio Managers with Delayed Tax Effects}\label{subSection numerical Merton no cons}

\subsubsection{CARA Case}\label{numerical results CARA no consump}

We consider now the problem of competition between portfolio managers with delay tax effects, Eq.~\eqref{XdyanmicsMertonNoCons}, in the CARA case where the rewards are given by Eqs.~\eqref{costMertonNoConsCARA1}--\eqref{a and Z_disc}. In our numerical experiment, we select the parameter values as shown in Table \ref{table param CARA no consump}.

For this problem, we lower the learning rate throughout training by taking it to be $10^{-2}$ for the first 500 rounds of DFP, $10^{-3}$ for the next 500, and $10^{-4}$ for the remaining 700 rounds. The impact of this training on the approximation of the true Nash equilibrium rewards is shown in Figure \ref{figure:TrainingEvol_NoConsCARA} as explained in Section \ref{cost graph method}.

\begin{comment}{Specifically, for our simulation, we use the training schedule given by Table \ref{table training schedule CARA no consump}.
\begin{table}[H]
\centering
\begin{tabular}{|c||c|c|c|}
\hline
Round of DFP & 0 to 500 & 501 to 1000 & 1001 to 1700 \\
\hline
Learning rate & 1e-2 & 1e-3 & 1e-4 \\
\hline
\end{tabular}
\caption{Training schedule for CARA case.} \label{table training schedule CARA no consump}
\end{table}}
\end{comment}

\begin{figure}[H]
    \centering
    \centerline{\includegraphics[width=.55\textwidth, trim={0 2em 0 0},clip]{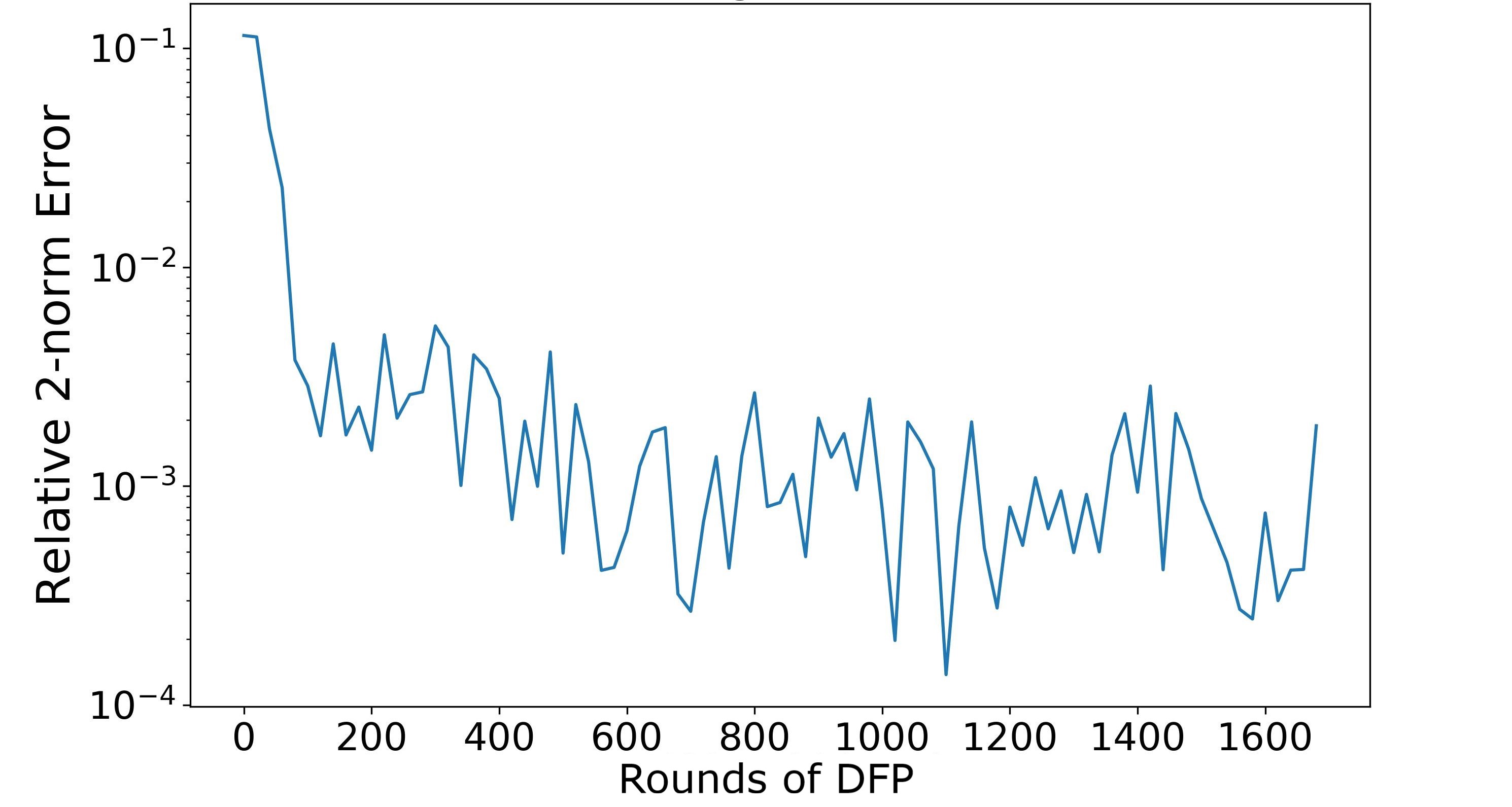}}
    \caption{The relative 2-norm error \eqref{2 norm relative error} over the course of training between $(\hat{J}^1, \cdots, \hat{J}^N)$ under the LSTM controls and the true Nash equilibrium controls for the CARA case. We take $N_\mathrm{batch} = 2^{15}$ in the computation of $(\hat{J}^1, \cdots, \hat{J}^N)$ according to Eq.~\eqref{general_numerical_SDDG_cost}. The length of training is measured in terms of rounds of DFP.}
    \label{figure:TrainingEvol_NoConsCARA}
\end{figure}

\noindent The decreasing relative $2$-norm error that plateaus at a level near $10^{-3}$ indicates a successful training of the controls. On one hand, this demonstrates that the rewards simulated under the LSTM controls are close to the true Nash equilibrium rewards. However, we are also interested to see how the trajectories themselves compare under both LSTM and true Nash equilibrium controls. Following the methodology in Section \ref{trajectory graph method}, we compare the trajectories under the LSTM controls to their true Nash equilibrium counterparts, and the resulting plots are shown in Figure \ref{figure:NoConsCARA_paths}.

\begin{figure}[H]
    \centering
    \centerline{\includegraphics[width=.8\textwidth, trim={0 0 0 0},clip]{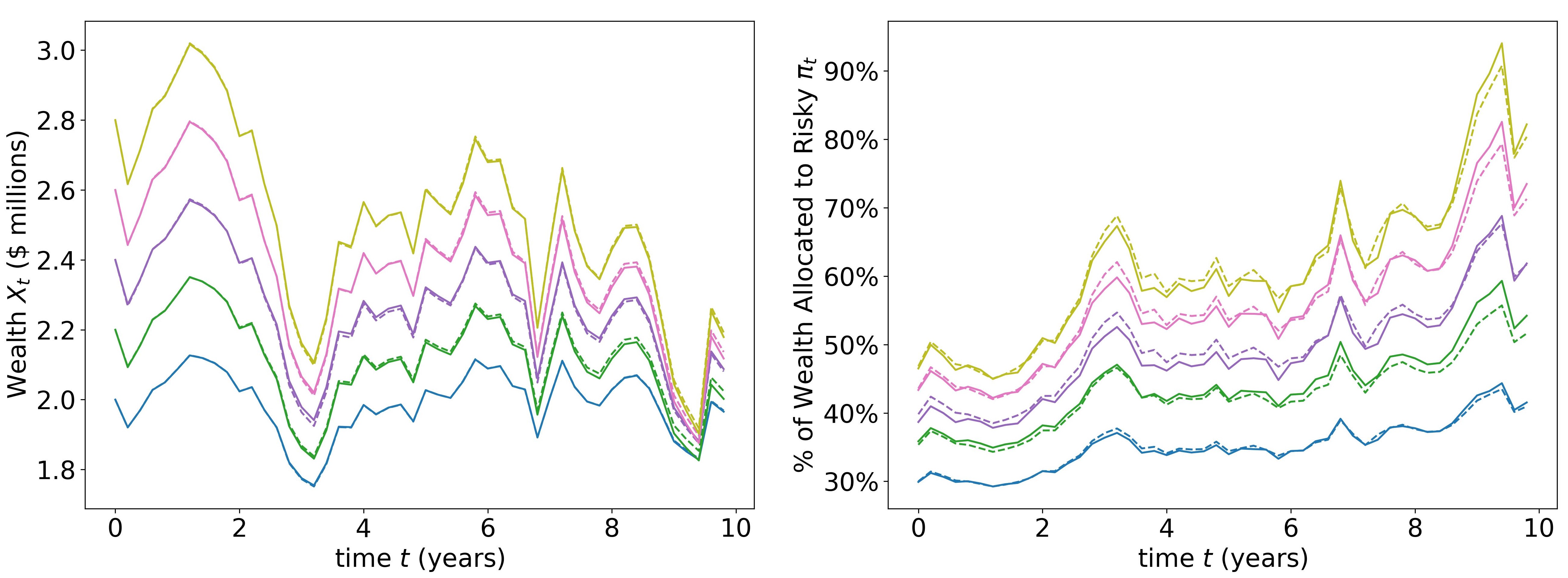}}
    \caption{Left: A sample path of the wealth processes for players  $1, 3, 5, 7,$ and $9$ (each player corresponds to a unique color) under the true Nash equilibrium controls (dashed) and the trained LSTM control (solid). Right: True Nash equilibrium controls (dashed) vs LSTM controls (solid). Controls represent the fraction of total wealth allocated to the risky asset at time $t$ for a given player. }
    \label{figure:NoConsCARA_paths}
\end{figure}

\noindent We see that the state process trajectories are nearly identical under both true and LSTM controls respectively. At the same time, the LSTM controls themselves are not only approximating the absolute level of control, but adapting to noises within the state process as illustrated by the LSTM controls matching the shape of the Nash equilibrium controls.

\subsubsection{CRRA Case}
We now consider the problem of competition between portfolio managers with delayed tax effects, Eq.~\eqref{XdyanmicsMertonNoCons}, in the CRRA case where the rewards are given by Eqs.~\eqref{costMertonNoConsCRRA1}--\eqref{costMertonNoConsCRRA2}. In our numerical experiments, we select the  parameters for this problem as shown in Table \ref{table param CRRA no consump}. For this problem, we lower the learning rate throughout training by taking it to be $10^{-2}$ for the first 500 rounds of DFP, $10^{-3}$ for the next 500, and $10^{-4}$ for the remaining 500 rounds. The approximation of the empirical rewards under the LSTM controls to that under the true Nash equilibrium controls (see Section \ref{cost graph method} for details) is examined by Figure \ref{figure:TrainingEvol_NoConsCRRA} below.

\begin{figure}[H]
    \centering
    \centerline{\includegraphics[width=.55\textwidth, trim={0 2em 0 0},clip]{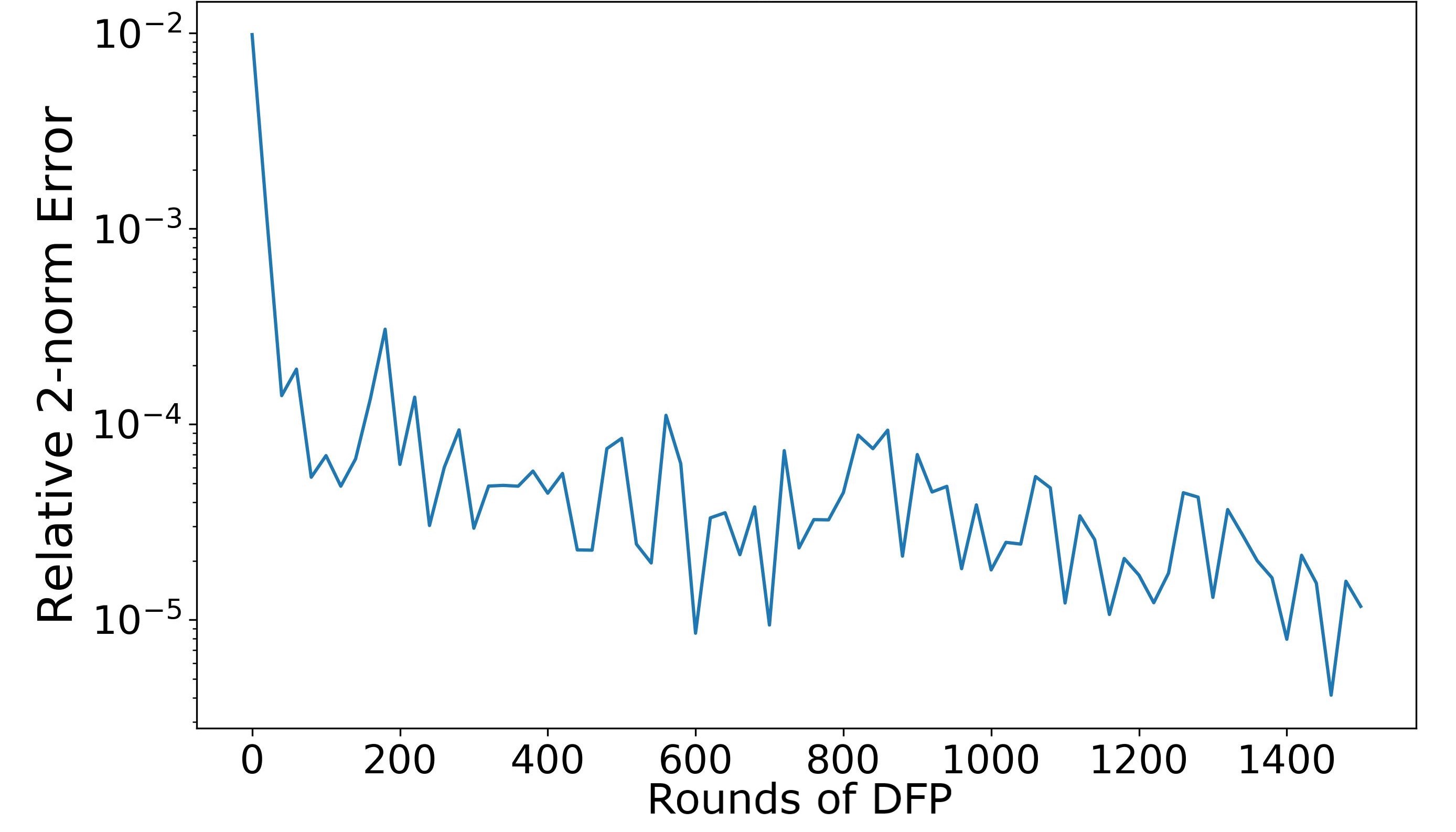}}
    \caption{The relative 2-norm error \eqref{2 norm relative error} over the course of training between $(\hat{J}^1, \cdots, \hat{J}^N)$ under the LSTM controls and the true Nash equilibrium controls for the CRRA case. We take $N_\mathrm{batch} = 2^{15}$ in the computation of $(\hat{J}^1, \cdots, \hat{J}^N)$ according to Eq.~\eqref{general_numerical_SDDG_cost}. The length of training is measured in terms of rounds of DFP.}
    \label{figure:TrainingEvol_NoConsCRRA}
\end{figure}

\noindent The relative $2$-norm error decreasing throughout training and plateauing at levels near $10^{-5}$ indicates that the training is successful and the rewards experienced under the LSTM controls are approximating the rewards experienced under the true Nash equilibrium controls. Figure \ref{fig:NoConsCRRA_paths} below allows us to compare the paths induced by these trained controls to their true Nash equilibrium counterparts.

\begin{figure}[H]
    \centering
\centerline{\includegraphics[width=.8\textwidth]{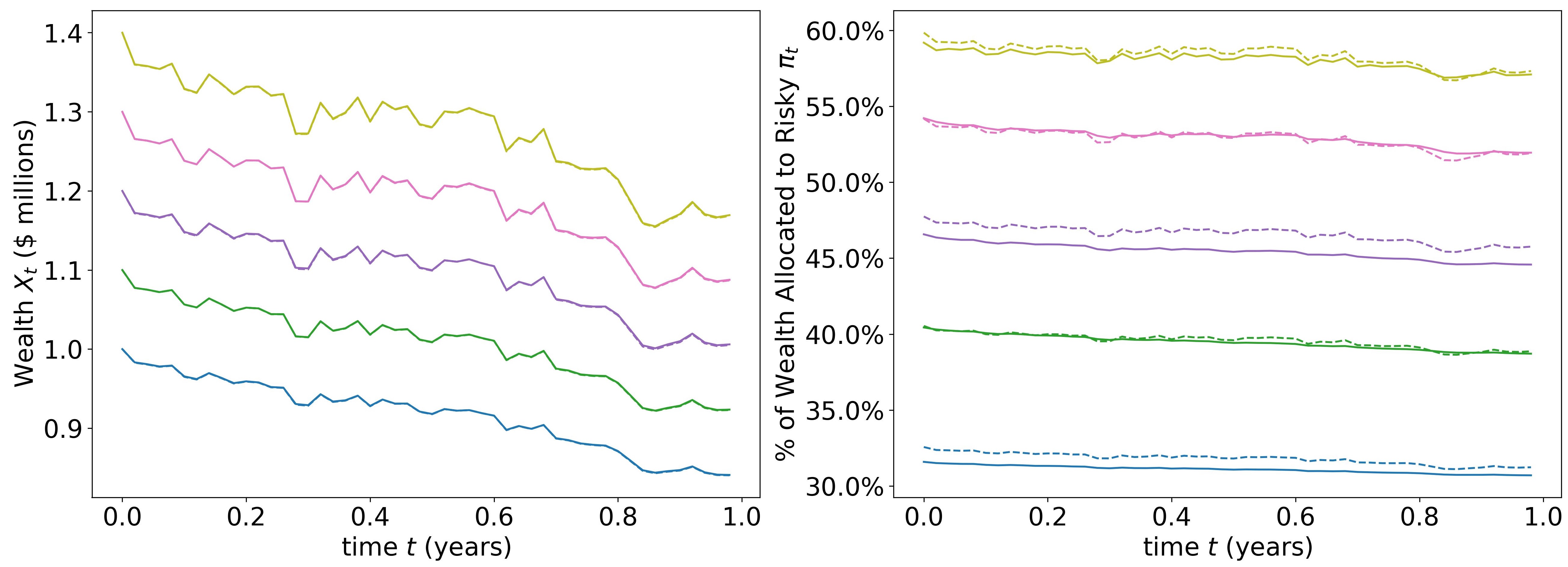}}
    \caption{Left: A sample path of the wealth processes for players  $1, 3, 5 ,7,$ and $9$ (each player corresponds to a unique color) under the true Nash equilibrium controls (dashed) and the trained LSTM control (solid). Right: True Nash equilibrium controls (dashed) vs LSTM controls (solid). Controls represent the fraction of total wealth allocated to the risky asset at time $t$ for a given player. }
    \label{fig:NoConsCRRA_paths}
\end{figure}

\noindent As indicated by Figure \ref{fig:NoConsCRRA_paths}, we see that the numerical methodology results in LSTM controls that accurately depict the true Nash equilibrium dynamics of the problem in question. The corresponding wealth processes under the true and LSTM controls are nearly identical, while the paths of the LSTM controls themselves are coinciding well with their true counterparts.

\subsection{Results for Consumption and Portfolio Allocation Game with Delayed Tax Effects}\label{subSection numerical Merton with cons}

We now consider the consumption and portfolio allocation game with delayed tax effects given by Eqs.~\eqref{XdyanmicsMertonwCons}--\eqref{costMertonwConsCRRA3}. For our numerical experiments, we select the value of the parameters as shown in Table \ref{table param CRRA w consump}. In this example, we use $10^{-2}$ as the learning rate for the first 500 rounds of DFP, $10^{-3}$ for the subsequent 500 rounds, and $10^{-4}$ for the final 1000 rounds. Following Section \ref{cost graph method}, the successive approximation of the Nash equilibrium rewards over training is illustrated by Figure \ref{figure:TrainingEvol_WConsCRRA}.
\begin{figure}[H]
    \centering
    \centerline{\includegraphics[width=.55\textwidth, trim={0 2em 0 0},clip]{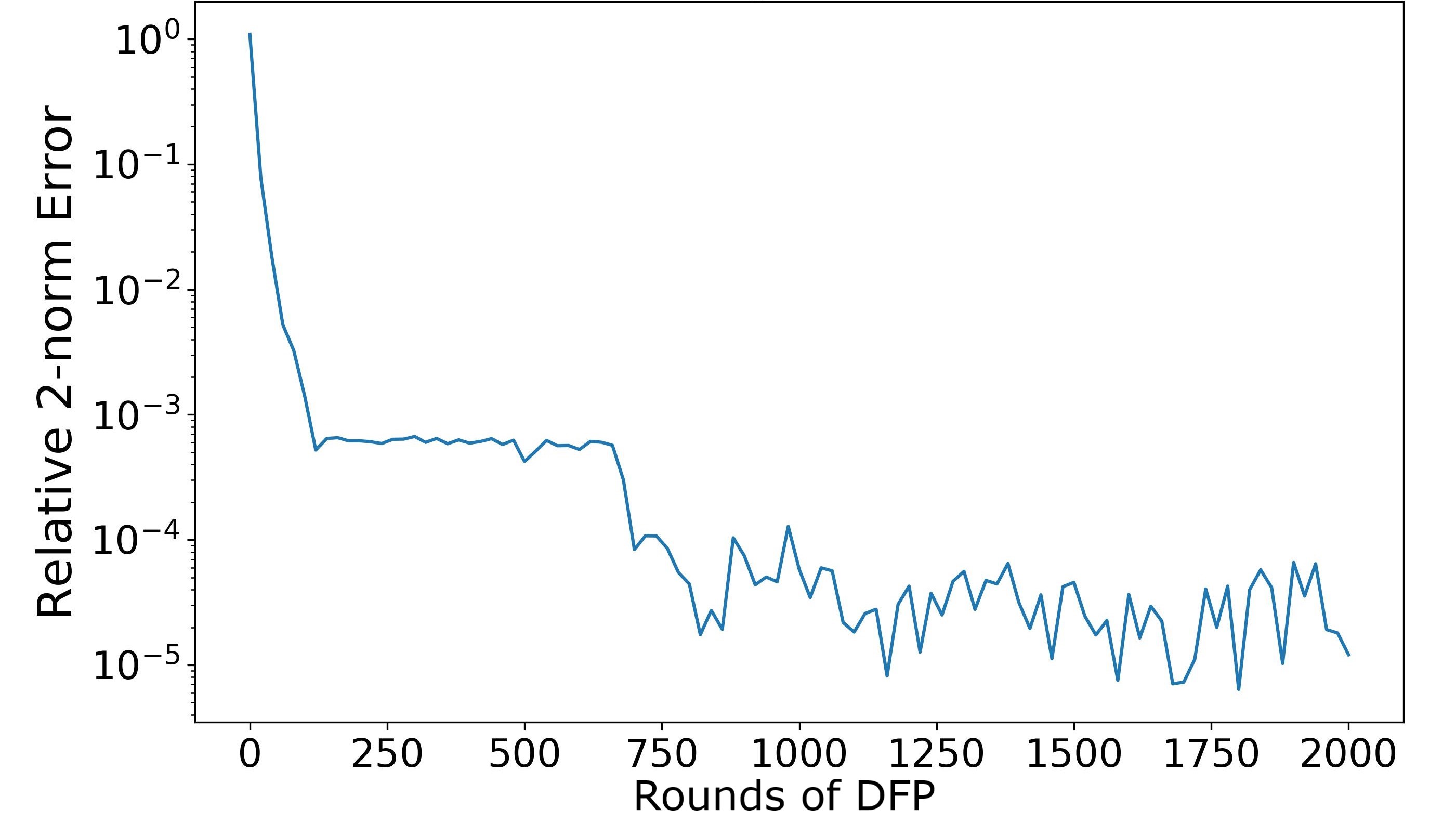}}
    \caption{The relative 2-norm error \eqref{2 norm relative error} over the course of training between $(\hat{J}^1, \cdots, \hat{J}^N)$ under the LSTM controls and the true Nash equilibrium controls. We take $N_\mathrm{batch} = 2^{15}$ in the computation of $(\hat{J}^1, \cdots, \hat{J}^N)$ according to Eq.~\eqref{general_numerical_SDDG_cost}. The length of training is measured in terms of rounds of DFP.}
    \label{figure:TrainingEvol_WConsCRRA}
\end{figure}

\noindent  This example showcases the importance of the training schedule. We notice that there is an initial plateau in the approximation of the true Nash equilibrium empirical rewards around the level of $10^{-3}$ relative $2$-norm error. The learning rate first changes after $500$ rounds of DFP and the plateau breaks shortly thereafter. The learning rate is again decreased at round $1000$ of DFP, yet another substantial decrease in relative error following this change is not seen. The relative $2$-norm error reaches a final plateau near $10^{-5}$ indicating a successful training.

The success of training is also reflected in the results from comparing the realized trajectories under both true and LSTM controls in Figure \ref{fig:WithConsCRRA_paths}. In this case, player $i$ has two controls representing the stock allocation at time $t$, $\pi^i_t$, as well as the consumption rate at time $t$, $c^i_t$. We plot in Figure \ref{fig:WithConsCRRA_paths} their corresponding unnormalized versions which are the wealth allocated to the stock and the annualized run rate of wealth consumed respectively.

\begin{figure}[H]
    \centering
    \centerline{\includegraphics[width=0.92\textwidth, height = 0.16\textheight]{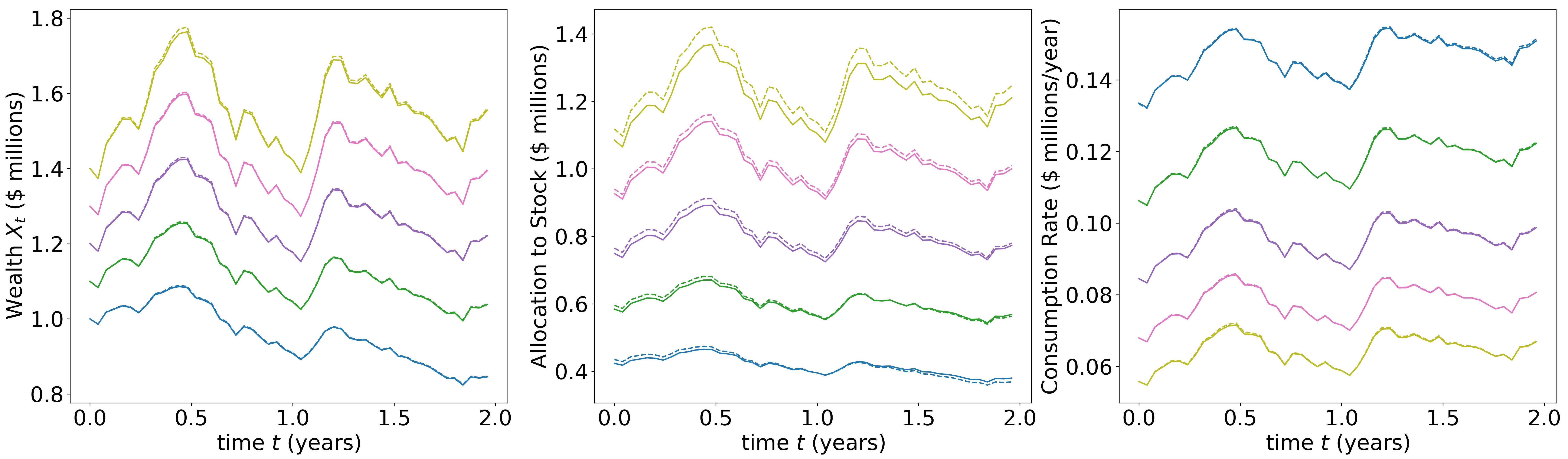}}
    \caption{Left: A sample path of the wealth processes of players $1, 3, 5 ,7,$ and $9$ under the true Nash equilibrium controls (dashed) and the trained LSTM control (solid) of the corresponding players. Center:  Nash equilibrium wealth allocation to stock (dashed) vs LSTM allocation (solid). Right: Nash equilibrium total consumption rate (dashed) vs LSTM consumption rate (solid).}
    \label{fig:WithConsCRRA_paths}
\end{figure}

\noindent We see that the trajectories produced by the trained LSTM controls coincide well with those produced under the true Nash equilibrium controls. This is especially apparent in the consumption control, which like the wealth process, contains trajectories that almost entirely overlap with the true Nash equilibrium dynamics. This example highlights the ability of the proposed algorithm to succeed in the important case where each player has multiple controls.

\subsection{Results for the Inter-Bank Lending Model}\label{subSection numerical inter-bank}
Lastly, we present the numerical results for the inter-bank lending model for systemic risk given by Eqs.~\eqref{iblwd_eqn}--\eqref{ibwld_loss}. The parameter choices are summarized in Table \ref{table iblwd params}. The learning rate iterated throughout the training is chosen to be $10^{-2}$ in the first 500 rounds of DFP, $10^{-3}$ in the next 500 rounds, $10^{-4}$ in the subsequent 500 rounds, and $10^{-5}$ in the last 2500 rounds. The relative 2-norm error between empirical rewards over training, in this case, is shown in Figure \ref{figure:TrainingEvol_IBLWD}.
\begin{comment}
\begin{table}[H]
\centering
\begin{tabular}{|c||c|c|c|c |}
\hline
Round of DFP & 0 to 500 & 501 to 1000 & 1001 to 1500 & 1501 to 4000 \\
\hline
Learning rate & 1e-2 & 1e-3 & 1e-4 & 1e-5 \\
\hline
\end{tabular}
\caption{Training schedule for inter-bank lending problem.} \label{table training schedule IBLWD}
\end{table}
\end{comment}
\begin{figure}[H]%try H or htpb
    \centering
    \centerline{\includegraphics[width=.55\textwidth, trim={0 1em 0 0},clip]{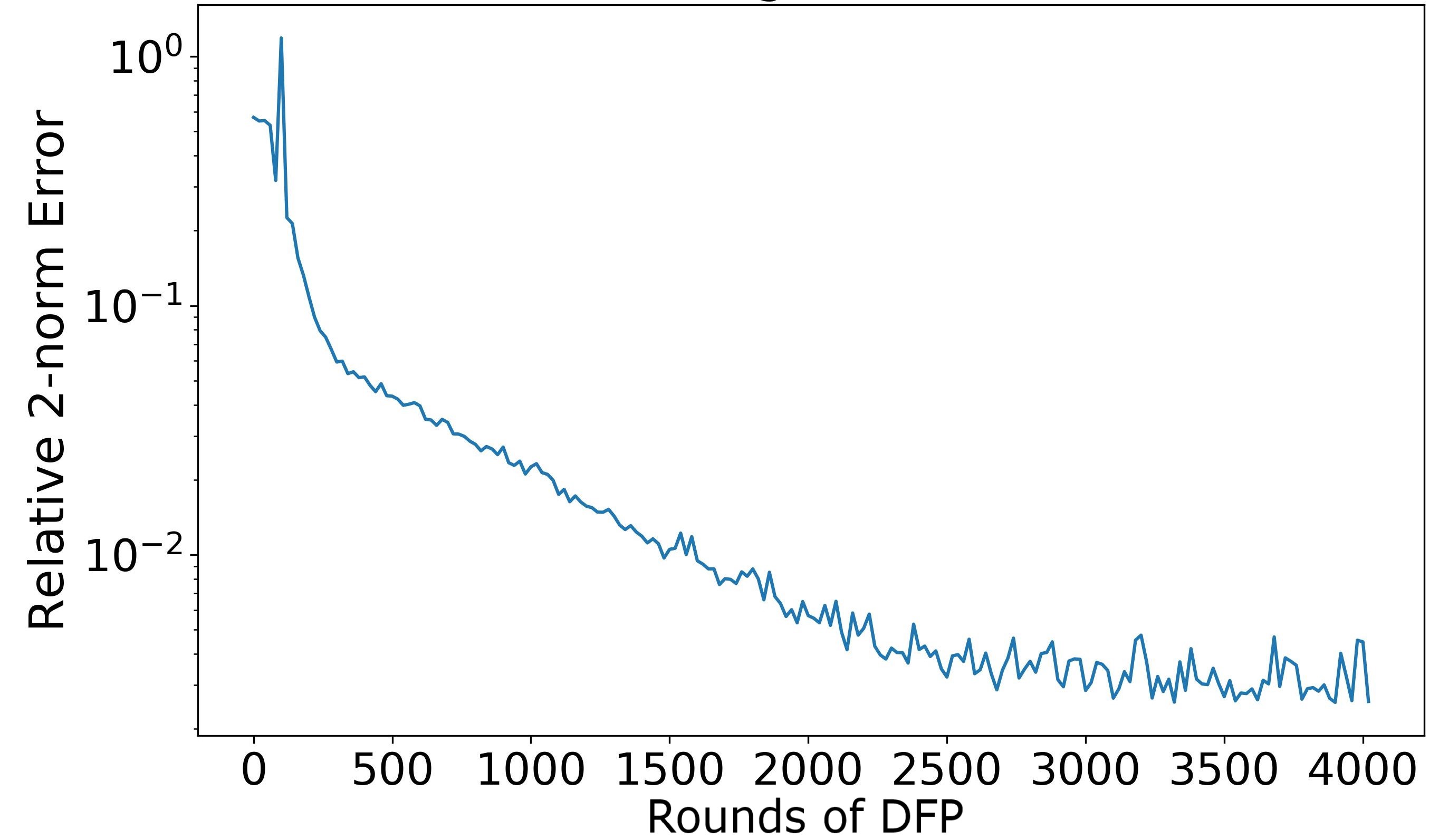}}
    \caption{The relative 2-norm error \eqref{2 norm relative error} over the course of training between $(\hat{J}^1, \cdots, \hat{J}^N)$ under the LSTM controls and the true Nash equilibrium controls for the inter-bank lending model. We take $N_\mathrm{batch} = 2^{12}$ in the computation of $(\hat{J}^1, \cdots, \hat{J}^N)$ according to Eq.~\eqref{general_numerical_SDDG_cost}. The length of training is measured in terms of rounds of DFP.}  %\rh{please change the x-label}
    \label{figure:TrainingEvol_IBLWD}
\end{figure}

\noindent The 2-norm relative error steadily decreases illustrating a continual improvement throughout training. It is important to note that in this case, what we call the ``true'' Nash equilibrium controls are actually themselves approximated as the true controls are given in terms of a PDE system (see Proposition \ref{thm iblwd}).\footnote{The PDE system is solved numerically with a discretization of 800 equally spaced slices for $t \in [0,T]$ and 50 slices for both  $s \in [-\tau,0]$  and $r \in [-\tau,0]$. The PDE is a nonlinear transport equation and the forward Euler scheme is used to solve the PDE.} This additional source of error from the PDE approximation is likely the cause of the increased smoothness of the curve in Figure \ref{figure:TrainingEvol_IBLWD} and the slightly higher levels of relative error compared to previous cases. Despite this additional source of error, the relative $2$-norm error between these two reaches levels close to $10^{-3}$, and the plateau of this error indicates successful training. However, the comparison of trajectories, Figure \ref{fig:iblwd_paths}, is the foremost illustration showcasing the success of the numerical algorithm for this particular example.

\begin{figure}[H]
    \centering
    \centerline{\includegraphics[width=.8\textwidth]{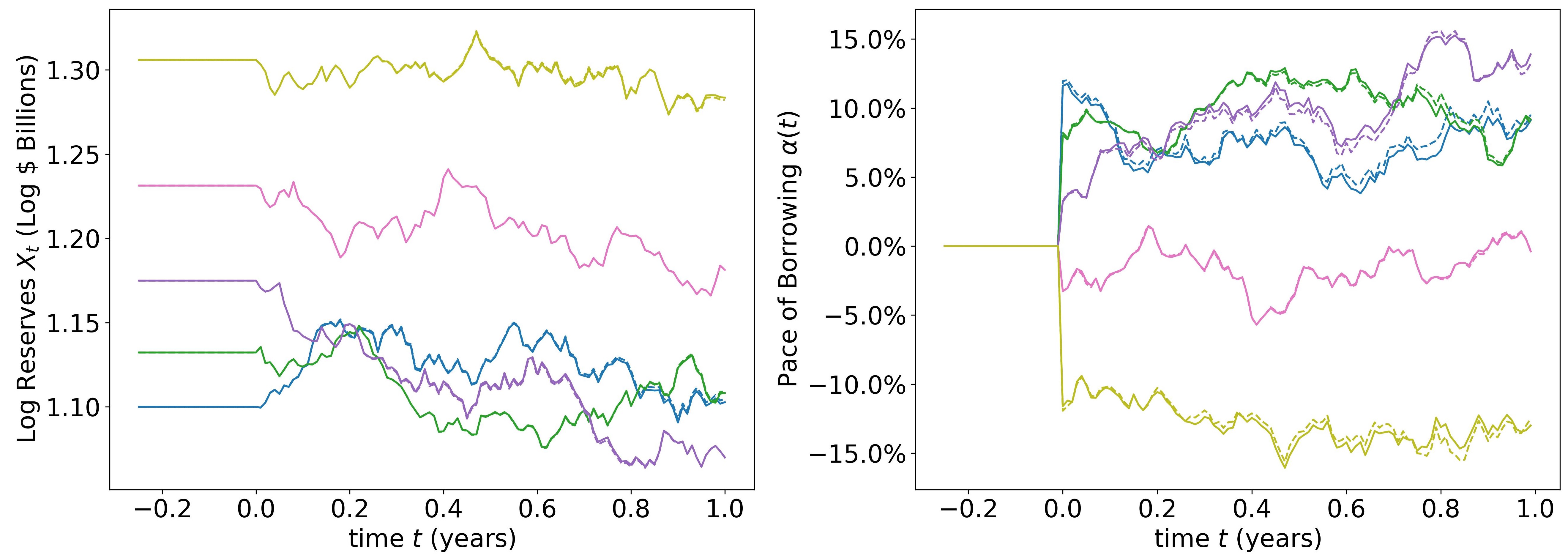}}
    \caption{Left: A sample path of the log-monetary reserves of banks $1, 3, 5 ,7,$ and $9$ under the true Nash equilibrium controls (dashed) and the trained LSTM control (solid). Right: The corresponding paths (with respect to the left picture) of Nash equilibrium controls (dashed) and LSTM controls (solid) for each bank. The controls represent the pace of borrowing/lending for each bank measured as the annualized run rate of borrowing as a percentage of current monetary reserves. }
    \label{fig:iblwd_paths}
\end{figure}

\noindent This is an interesting and unique case in that this problem's dynamics involved delay with respect to the controls themselves rather than the state process. Despite this complexity, we see that the LSTM controls are successful in approximating the true Nash equilibrium control trajectories by conforming to their shape and absolute levels exceptionally well, causing their induced state processes to be nearly identical to that of the true Nash equilibrium.

\section{Conclusion} \label{Section conclusion}
Finding the closed-loop Nash equilibrium for stochastic delay differential games poses challenges due to the inherent high dimensionality of the problem. To address this numerically, we propose a deep learning algorithm that involves parameterizing the controls with LSTM recurrent networks and simulating the trajectories of the discretized dynamical system under these LSTM controls to approximate the expected cost for each player. By approximating the expected costs given a choice of the LSTM-based controls, the LSTMs are trained in an iterative process inspired by the concept of fictitious play.

Separately, we have introduced and analyzed a new class of stochastic delay differential games that arise in finance and have obtained their analytical solutions. The derivation, interpretation, and proofs for these problems are presented in Appendices \ref{Appendix Intuition} and \ref{appendix_proof}.

From the numerical experiments we conducted, we have observed compelling evidence that our proposed algorithm successfully approximates the true Nash equilibrium controls. The trajectories of the trained controls closely match those of the true controls after training. Furthermore, we have observed the successful approximation of the costs (or rewards) for each player under the LSTM controls towards the true Nash equilibrium costs (or rewards) during the course of training. This is demonstrated by a decrease in the relative 2-norm error between these vectors throughout the training process, with relative errors reaching values lower than $10^{-5}$ in some cases.

\section{Acknowledgements}

R.H. was partially supported by the NSF grant DMS-1953035, and the Early Career Faculty Acceleration funding and the Regents’ Junior Faculty Fellowship at the University of California, Santa Barbara.

\bibliographystyle{plain}

\bibliography{ref}

\appendix

\section{Interpretation and Derivations of the Games in Sections~\ref{subSection the competition of portfolio managers with Delayed Tax Effects}--\ref{subSection the Competitive Consumption and Portfolio Allocation with Delayed Tax Effects problem}}\label{Appendix Intuition}

\subsection{Portfolio Optimization with Delayed Taxes Effects}\label{appendix Portfolio Optimization with Delayed Taxes}
%We summarize a certain case of the problem originated by Merton in \cite{MertonOriginal} in order to motivate future notation choices as well as to provide an introductory basis for readers.

We start with restating the original Merton problem \cite{MertonOriginal}.  Consider an investor who chooses between a stock and a bond. At time $t$, the fraction of wealth $\pi_t$ is invested in the stock, and $1 - \pi_t$ is invested in the bond. In general, we have $\pi_t \in \mathbb{R}$, where $\pi_t > 1$ corresponds to a leveraged stock position, while $\pi_t < 0$, corresponds to the investor being short the stock. The bond is assumed to accrue at a continuously compounded rate of interest $r$ and the stock evolves according to the Black-Scholes model $\frac{\ud S_t}{S_t} = \mu \ud t + \sigma \ud W_t$. 
%with mean return $\mu$ and volatility $\sigma$. 
Denoting  by  $X_t$ the investor's wealth at time $t$, one then has
\begin{comment}
notes that the portfolio return, or the return on the total wealth of the investor, $\frac{\ud X_t}{X_t}$ is the weighted sum of the stock return $\frac{\ud S_t}{S_t}$ and cash return $r \ud t$ weighted by the portfolio $(\pi_t,1-\pi_t)$, where $(S_t)$ is the stock price process. This gives the dynamics for the wealth of the investor given by
\begin{equation}\label{portfolioreturn}
     \frac{\ud X_t}{X_t} = \pi_t \frac{\ud   S_t}{ S_t} + (1 - \pi_t ) r \ud t, \ t \in (0,T].
\end{equation}
Under the Black-Scholes model with mean return $\mu$ and volatility $\sigma$, we have $\frac{\ud S_t}{S_t} = \mu \ud t + \sigma \ud W_t$. Substituting these dynamics of the stock $(S_t)$ into Eq.~\eqref{portfolioreturn} and rearranging terms, we arrive at the controlled dynamical system
%Of course we have to specify a model for the stock price process $(S_t)$. A simple and commonly used assumption is to assume that the stock returns are stationary and independent. One such example would be the Black-Scholes model, we have that the return of the stock over disjoint intervals of the form $[t,t + \ud t]$ are i.i.d. and normally distributed according to some mean $\mu \ud t$ and variance $\sigma^2 \ud t$. This gives that the stock price $(S_t)$ evolves according to the SDE
\end{comment}
\begin{equation}\label{intro_Merton_problem}
     \ud X_t = [(\mu - r) \pi_t X_t  + r X_t]  \ud t + \sigma \pi_t X_t \ud W_t, \ t \in (0,T].
\end{equation}
The investor aims to choose a strategy $\pi$ to optimize her expected utility of terminal wealth 
\begin{equation}\label{intro_Merton_problem_reward}
    J[\pi] = \mathbb{E}[U(X_T)],
\end{equation}
subject to the dynamics \eqref{intro_Merton_problem}. Intuitively, the expected utility quantifies the investor's desire for the random outcome $X_T$. 

%The framing of the problem as one of expected utility is natural as the classical work of von Neumann and Morgenstern \cite{vonneumann1947} shows how the existence of preferences of a rational agent among random monetary payouts or ``lotteries''  is equivalent to the existence of a utility function with which the agent ranks lotteries according to their expected utilities.

%The model we now consider is identical in structure to the portfolio optimization with complete memory problem introduced in \cite{delayedproblem}, which features both consumption and portfolio allocation controls. For now, we will ignore the consumption control and focus on the new consequence of delayed taxes. In this model, the Merton problem introduced in Appendix \ref{Appendix An Introduction to the Portfolio Optimization Problem of Merton} 

We now consider an additional outflow of wealth due to taxes in Eq.~\eqref{intro_Merton_problem}. We assume that at the end of the period $[t,t+\mathrm{d}t]$, the investor pays the amount $\mu_2 Y_t \ud t$ in taxes, where $Y_t = \int_{-\infty}^t \lambda e^{-\lambda(t- s) } X_{s} \ud s$ is the investor's exponentially averaged past wealth. This leads to the modified dynamics of the wealth process $X$ given by
\begin{equation}\label{delayed_Merton_problem}
 \ud X_t =  \left((\mu_1 - r) \pi_t X_t  + rX_t - \mu_2 Y_t   \right) \ud t + \sigma \pi_t X_t \ud W_t, \ t \in (0,T].
\end{equation}
This model with $\mu_2 <0$ was introduced and solved in \cite{delayedproblem} arising from a type of momentum effect. For our consideration, we take $\mu_2 > 0$ and the term $\mu_2 Y_t \ud t$ represents the fact that the investor is paying the fixed percentage (or tax rate) $\mu_2 > 0$ of their historical wealth. This outflow could represent management fees, trading fees, and/or taxes. For simplicity, we shall refer to it as ``taxes'' in the sequel. 
%will group all these fees together under the label of ``taxes'' to avoid repetition.

Such modeling enables us to capture some realistic features:  1) taxes increase with and proportional to wealth; 2) there is a delay between when a tax is realized and when it is paid; 3) this delayed period for taxes varies for a given tax and is itself random. To see this, let's assume that the taxes paid over $[t,t + \mathrm{d}t]$ occur due to numerous tax bills that were realized $\tau$ units in the past. If we further assume that the time to pay these taxes, $\tau$, follows an exponential distribution with a rate $\lambda$, then one could approximate the taxes paid in $[t,t + \mathrm{d}t]$ by its mean under the scenario of a high frequency of tax occurrences with small tax amounts at each occurrence. This gives rise to the flux of wealth of $ -\mu_2 \int_{-\infty}^t \lambda e^{-\lambda(t- s)} X_{s} \ud s \ud t =  -\mu_2 Y_t \ud t$ as it appears in Eq.~\eqref{delayed_Merton_problem}. In essence, the tax at time $t$ of $\mu_2 Y_t \ud t$ naturally represents the delayed accrual of taxes due at time $t$ based on past wealth as a result of a delay in billings.

%It is a reasonable assumption to take that investors pay taxes proportional to their current wealth. While this does not capture certain intricacies of more complicated tax structures, it does capture the simple notion that those with more wealth will pay more in taxes. However, at the same time, one does not always have to pay the taxes incurred at time $t$ immediately. One may be able to pay their taxes some amount of time $\tau$ in the future. In this case, if the investor must pay at time $t+\tau$ a tax proportional to their current assets at time $t$, then this would result in a pointwise delay in the model \eqref{intro_Merton_problem} occurring through an additional flux $-\mu_2 X_{t-\tau} \ud t$. 

%The reason that we consider a nontrivial tax structure depending on the entire history of wealth is to better capture realistic features. The delayed tax-structure based on exponentially average past wealth can be seen as arising from frequent billings of taxes proportional to one's current wealth, but with the due dates of those billings occurring in future dates with exponentially distributed arrival times. 

Since Eq.~\eqref{delayed_Merton_problem} includes an outflow due to taxes, one observes that cash will no longer grow at the risk-free rate $r$. Therefore, it is natural to ask if there is a tax-adjusted risk-free rate that better represents the growth of cash in this model. This is addressed in Appendix \ref{appendix effective derivation}. Moreover, while $X_t$ represents the wealth of an investor at time $t$, the investor is carrying around a hidden tax liability given through their past history of wealth. We will also argue in Appendix \ref{appendix effective derivation} that the variable $Z_T = X_T + a Y_T$ represents the tax-adjusted wealth of the investor at time $t$ when $a$ is given by $a = \frac{-(r+\lambda) + \sqrt{(r+\lambda)^2 - 4 \lambda \mu_2}}{2\lambda}$. With this in mind, the problem we consider is that of an investor who seeks to maximize the quantity
%Even though the model ends at time $T$, it is reasonable to assume that the same tax dynamics are still intact beyond the terminal time for the endowment $X_T$. In other words, the variable $Y_T$ will continue to affect the endowment value at future times. In this case, the investor would not seek to maximize the expected utility of their wealth $X_T$, but instead that of their tax-adjusted wealth $Z_T$. Therefore, the problem we consider is that of an investor who chooses a portfolio strategy $\pi$ such that with respect to the wealth dynamics \eqref{delayed_Merton_problem}, they seek to maximize the quantity
\begin{equation}\label{delayed_Merton_problem_reward}
   J[\pi] = \mathbb{E}[U(Z_T)],
\end{equation}
which is the expected utility of tax-adjusted terminal wealth. %The investor considers their tax-adjusted terminal wealth in the case that the same tax dynamics are still intact beyond the terminal time for whoever inherits the endowment $X_T$. In this case, one must continue to acknowledge the taxes that have been realized but are yet to be paid, which makes $X_T$ an overstatement of true wealth, while $Z_T$ accurately captures this quality.

Lastly, we mention that the stochastic control problem with such delay structure~\eqref{delayed_Merton_problem} and $\mu_2 <0$ was considered in \cite{delayedproblem}, including notably the form of the utility \eqref{delayed_Merton_problem_reward} depending on $Z_T = X_T + a Y_T$, where $a = \frac{-(r+\lambda) + \sqrt{(r+\lambda)^2 - 4 \lambda \mu_2}}{2\lambda}$. Therein, $\mu_2 < 0$ is interpreted as a momentum-like effect arising from the market structure. Our work considers $\mu_2>0$, leading to a quite different interpretation as discussed above.  
%the work here motivated the form of the model problem and inspired portions of our solution methodology featured in Appendix \ref{appendix_proof}. We take a different interpretation to \cite{delayedproblem} as it allows us to derive a precise meaning to the future utility functions we consider.

\subsection{Tax-Adjusted Wealth and the Tax-Adjusted Risk-Free Rate}\label{appendix effective derivation}

We have discussed the motivation and interpretation of Eqs.~\eqref{delayed_Merton_problem}--\eqref{delayed_Merton_problem_reward}, which comes from an investor who is paying taxes at time $t$ at a rate $\mu_2>0$ on her exponential average of past wealth. In this section, we give further interpretations to the quantities $Z_t$ and $r + \lambda a$, as the tax-adjusted wealth and the tax-adjusted risk-free rate respectively. We define the tax-adjusted risk-free rate to be the long-term exponential growth rate of wealth for an all-bond account. Then tax-adjusted wealth is the process that grows precisely at the tax-adjust risk-free rate when considering the all-bond investor. In other words, $X_t$ is no longer the best measurement of an investor's ``true'' wealth as it does not incorporate the hidden tax liabilities which arise due to the past history of $X_t$, yet contributes to taxes beyond time $t$; and the usual risk-free rate $r$ no longer represents the rate of accrual of a pure bond account due to tax drag.

To see this, we consider the dynamics of an all-bond investor $\pi_t \equiv 0$:
\begin{equation}\label{eqn x}
 \ud X_t = rX - \mu_2 Y_t   \ud t .
\end{equation}
Recall that  $Y_t = \int_{-\infty}^t \lambda e^{-\lambda(t- s) } X_{s} \ud s$ and $\ud Y_t = \lambda (X_t -  Y_t)   \ud t$, one has for arbitrary $a \in \mathbb{R}$
\begin{equation*}
 \ud (X_t + aY_t) = (r + \lambda a)X_t \ud t + (-\mu_2 - \lambda a) Y_t   \ud t.
\end{equation*}
Therefore if $a$ satisfies $-\mu_2 - \lambda a = a(r + \lambda a)$, then we will have
\begin{equation*}
 \ud (X_t + aY_t) = (r + \lambda a)(X_t + aY_t) \ud t ,
 \end{equation*}
which occurs when $a$ takes values of 
$a_{\pm} = \frac{-(r+\lambda) \pm \sqrt{(r+\lambda)^2 - 4 \lambda \mu_2}}{2\lambda}$.
Note that $a_{\pm}$ are distinct real numbers since $(r+\lambda)^2 - 4 \lambda \mu_2>0$ is assumed.

Denoting $Z^{\pm}_t = X_t + a_{\pm}Y_t $, we have two linearly independent representations for $X$, which allows us to eliminate $Y$ resulting in an expression of $X$ as a linear combination of $Z^+$ and $Z^-$.  Using that $Z^{\pm}_t = Z^{\pm}_0 e^{(r+\lambda a_\pm)t}$, we obtain the expression
\begin{equation*}
 X_t = c_+ e^{(r+\lambda a_+)t} + c_- e^{(r+\lambda a_-)t},
\end{equation*}
where $c_+ = \frac{a_-}{a_- - a_+} (X_0 + a_+ Y_0)$ and $c_- = \frac{-a_+}{a_- - a_+} (X_0 + a_- Y_0)$. Since $\lambda > 0$ and $a_+ > a_-$, we have that $r + \lambda a_+ > r + \lambda a_-$. Because of this, the wealth of the all-bond investor, $X_t$, has the property
\begin{equation*}
\lim_{t \rightarrow \infty} e^{-kt} X_t =
    \begin{cases}
        \infty, &  k <  r + \lambda a_+,\\
        c_{+}, &  k = r + \lambda a_+, \\
        0 ,&  k > r + \lambda a_+, \\
    \end{cases}
\end{equation*}
where $c_+>0$ holds assuming the initial quantity $X_0 + a_+ Y_0 > 0$ as well as $r+\lambda > 0$. This is guaranteed under either of the realistic assumptions that $r \geq 0$ or $|r| << \lambda$. Therefore, the rate $r + \lambda a_+$ is precisely the long-run exponential growth rate of an all-bond account, meaning cash grows at the rate  $r + \lambda a_+$ in the long run. 
%Therefore, taking $a = a_{+} = \frac{-(r+\lambda) + \sqrt{(r+\lambda)^2 - 4 \lambda \mu_2}}{2\lambda}$, we see that the quantity $r + \lambda a$ is our defined tax-adjusted risk-free rate. 
And for this choice of $a$, one has % it follows that the cash only investor satisfies the relation
$
X_t + aY_t = (X_0 + aY_0)e^{(r+\lambda a)t}$, consequently, $X_t + aY_t$ can be interpreted as tax-adjusted wealth. Note that, in fact,  $c(X_t + aY_t)$ for any $c \in \mathbb{R}$ could represent the tax-adjusted wealth by our requirement.  However, $X_t + aY_t$ is the correct choice out of these by imposing a natural second requirement: the tax-adjusted wealth should be consistent with the wealth if the investor has no tax liability. Thus $c = 1$. 
%For example, consider an investor with $0$ tax liability at time $0$, but with positive wealth. In other words, the investor has $X_0  > 0$ and $X_t = 0$ for $t<0$. In this case, $Y_0 = 0$ as $X_t = 0$ for $t<0$, and so $c(X_0 + aY_0) = X_0$ only if $c=1$.
The notion of $X_t + aY_t$ as the tax-adjusted wealth also makes sense intuitively as in our case of taxes ($\mu_2 > 0$) results in $a<0$ under the realistic assumption that $r+\lambda > 0$. Hence $X_t + aY_t$ represents an adjustment to total assets $X_t$ taking into account the tax liability given through $Y_t$.

%This observation allows us to ask if there is a reasonable characterization for a tax-adjusted wealth and tax-adjusted risk-free rate, and if so, what could capture these expressions? We define the tax-adjusted risk-free rate to be the long-term exponential growth rate of wealth for an all cash account. That is, the rate at which cash``effectively'' grows given a long enough time horizon. Then, assuming that one can find a tax-adjusted risk-free rate, we define the tax-adjusted wealth process through a key property. The tax-adjusted wealth process when considered for the all cash investor would satisfy the property that it grows precisely at the tax-adjust risk-free rate. 

\subsection{Competition between Portfolio Managers}\label{AppendixMertonGame}
We temporarily ignore the delay arising from taxes and review the game extension of Merton's original problem which has been considered in 
%posed in Eqs~\eqref{intro_Merton_problem},\eqref{intro_Merton_problem_reward}. This Merton problem has been extended to a game and is considered in 
\cite{Competition_among_Port_Managers, lackerwconsump, lackerNoConsump}.
We will briefly summarize them for convenience as they inspire the form of the new problems we consider in Sections~\ref{subSection the competition of portfolio managers with Delayed Tax Effects}--\ref{subSection the Competitive Consumption and Portfolio Allocation with Delayed Tax Effects problem}. The motivation \cite{Competition_among_Port_Managers} comes from competing portfolio managers who are selected by their clients (or awarded bonuses) not only based on the fund's absolute performance, but also based on the fund's performance compared to similar funds. To address this possibility, the portfolio manager may have a utility function that takes into account both their absolute performance, as well as their performance relative to their peer group. \cite{Competition_among_Port_Managers, lackerwconsump, lackerNoConsump} model this by considering the interaction between managers occurring through the reward function. 

Let $X_t^i$ be manager $i$'s wealth process. Her utility can depend on both her terminal wealth  $X^i_T$ as well as her terminal wealth relative to that of her peers $X^i_T - \olsi{X}_T$, where $\olsi{X}_T = \frac{1}{N} \sum_{i=1}^N X^i_T$ is the arithmetic mean. A weighted average $(1-\theta_i) X^i_T + \theta_i (X^i_T - \olsi{X}_T) =  X^i_T - \theta_i \olsi{X}_T $ can serve as the input for the utility function, where $\theta_i \in (0,1)$ measures the extent that manager $i$ weighs relative versus absolute performance. Specifically, \cite{lackerNoConsump} consider  the constant absolute risk aversion (CARA) case, and the reward for player $i$ is given by
\begin{equation}\label{CARA Utility}
    J^i = E [U_i( X^i_T - \theta_i \olsi{X}_T)],     \quad 
    U_i(z) = - \exp(-\frac{1}{\delta_i} z), \\
\end{equation}
where $\delta_i > 0$ is the risk tolerance of manager $i$. They consider constant relative risk aversion (CRRA) utilities
\begin{equation}\label{CRRA Utility}
    U_i(z) = \begin{cases}  
        \frac{1}{1- \frac{1}{\delta_i}} z^{1- \frac{1}{\delta_i}}, & \delta_i > 0, \delta_i \neq 1, \\
        \log(z), &  \delta_i = 1 ,\\
    \end{cases}
\end{equation}
where $\delta_i > 0$ is the risk tolerance of manager $i$. 

The CRRA case is also considered in \cite{lackerNoConsump}. In this case, for tractability, the competing managers compare relative performance to absolute performance modeled by the ratio $\frac{X_T^i}{(\olsi{X}_T)^{\theta_i}}$, where $\olsi{X} =  (\prod_{i=1}^N X^i)^{1/N}$ is the geometric average. Then we can say manager $i$ seeks to maximize her expected utility given by
\begin{equation}\label{CRRA Expected Utility}
J^i = E [U_i( X^i_T (\olsi{X}_T)^{-\theta_i})].
\end{equation}

\subsection{Competition between Portfolio Managers with Delayed Tax Effects} \label{appendix competition of portfolio managers with Delayed Tax Effects}
We have now discussed the extension of the Merton problem to one with delay as well as one as a game. Now, we seek to combine both aspects together. The wealth dynamics of manager $i$, denoted by $X_t^i$,  can easily be generalized via Eq.~\eqref{delayed_Merton_problem}, and given by
\begin{equation}\label{XdyanmicsMertonNoConsappendix}
\begin{aligned}
 & \ud X^i_t =  \left((\mu_1 - r) \pi^i_t X^i_t  + rX^i_t -\mu_2 Y^i_t  \right) \ud t + \sigma \pi^i_t X^i_t \ud W_t, \ t \in (0,T], \\
\end{aligned}
\end{equation}
where $Y^i_t = \int_{-\infty}^t \lambda e^{-\lambda(t- s)} X^i_{s} \ud s$. 

As for the reward function, we again consider both the CARA case \eqref{CARA Utility} and the CRRA case \eqref{CRRA Utility}, but replace the terminal wealth with the discounted, tax-adjusted terminal wealth. To recall the discussion in Appendix \ref{appendix Portfolio Optimization with Delayed Taxes} and \ref{appendix effective derivation}, we have identified the tax-adjusted risk-free rate to be $r+\lambda a$ and the tax-adjusted wealth for player $i$ at time $t$ to be 
\begin{equation*}
    Z^i_t = X^i_t + aY^i_t,
\end{equation*}
where $a = \frac{-(r+\lambda) + \sqrt{(r+\lambda)^2 - 4 \lambda \mu_2}}{2\lambda}$. The discounted, tax-adjusted wealth at time $t$ is then given by
\begin{equation*}
    Z^i_{disc,t} = e^{-(r+\lambda a)t} Z^i_t,
\end{equation*}
which discounts the tax-adjusted wealth back to time $0$ under the tax-adjusted risk-free rate. Therefore,  in the CARA utility case, it is natural to consider the reward for player $i$ by
\begin{equation*}
\begin{aligned}
   J^i[\bm \pi] = \mathbb{E} \left[U_i \left( Z^i_{disc,T} - \theta_i \olsi{Z}_{disc,T}  \right) \right], \\
\end{aligned}
\end{equation*}
where $\olsi{Z}_{disc,T}= \frac{1}{N} \sum_{i=1}^N Z^i_{disc,T}$. In contrast, for the CRRA utility case with $U_i$ given by Eq.~\eqref{CRRA Utility}, one has
\begin{equation*}
   J^i[\bm \pi] = \mathbb{E} \left[U_i \left(  Z^i_{disc,T} \olsi{Z}_{disc,T}^{-\theta_i}  \right)  \right],
\end{equation*}
where in this case $\olsi{Z}_{disc,T} =   \left(\prod_{i=1}^N Z^i_{disc,T}  \right)^{1/N}$.

In both cases, we see that the portfolio manager is simply comparing her terminal tax-adjusted wealth to that of her peers in the manner discussed in Appendix \ref{AppendixMertonGame} in determining her utility. 

\subsection{Consumption and Portfolio Allocation Game with Delayed Tax Effects}\label{appendix the Competitive Consumption and Portfolio Allocation with Delayed Tax Effects problem}

We further extend the modeling \eqref{XdyanmicsMertonNoConsappendix} by considering consumption. In addition to the investment strategy $\pi_t^i$, the consumption rate $c_t^i$ will also be chosen by investor $i$. The consumption rate $c_t^i$ measures investor $i$'s annual run rate of consumption at time $t$ as a fraction of her wealth. Precisely, over $[t,t+\mathrm{d} t]$ investor $i$ consumes the dollar amount given by $c^i_t X^i_t \ud t$. By including this outflow due to consumption, we have the wealth dynamics for player $i$ given by
\begin{equation*}
  \ud X^i_t = \Big((\mu_1 - r) \pi^i_t X^i_t  + rX^i_t -\mu_2 Y^i_t - c^i_t X^i_t \Big) \ud t + \sigma \pi^i_t X^i_t \ud W_t, \ t \in (0,T], \\
\end{equation*}
where as before, $Y^i_t = \int_{-\infty}^t \lambda e^{-\lambda(t- s)} X^i_{s} \ud s$. The quantity $Z^i_T = X^i_T + a Y^i_T$ with $a = \frac{-(r+\lambda) + \sqrt{(r+\lambda)^2 - 4 \lambda \mu_2}}{2\lambda}$ again represents the tax-adjusted wealth of player $i$. 

The expected utility will be a sum of the utility from discounted consumption and the utility from discounted terminal wealth. The exact form is motivated from \cite{lackerwconsump}, which has the utility of consumption given by $U^i((c_t^iX^i_t)(\olsi{cX}_t)^{-\theta_i})$ with $\olsi{cX}_t = (\Pi_{i=1}^N (c_t^iX_t^i))^{1/N}$. With the additional delayed tax effects in our modeling, 
%However, in \cite{lackerwconsump}, the risk-free rate is taken to be $0$, and there is no delayed tax consideration. This is 
we consider an analogous utility of discounted consumption and an analogous utility of discounted, tax-adjusted terminal wealth, where the discounting is taken with respect to the tax-adjusted risk-free rate. The resulting reward function for each player $i \in \{ 1, \cdots, N \}$ is given by
\begin{equation*}
  J^i[\bm \pi, \bm c] = E\Big[\int_0^T U^i(C_{disc,t}^i \olsi{C_{disc,t} }^{-\theta_i}) \ud t + \epsilon_i U_i\Big( Z_{disc,T}^i \olsi{Z_{disc,T}}^{-\theta_i} \Big) \Big],
\end{equation*}
where
\begin{equation*}
    \begin{aligned}
   &  C_{disc,t}^i =  e^{-(r+\lambda a)t} c^i_t X^i_t, & \olsi{C_{disc,t}} = \big(\prod_{i=1}^N C^i_{disc,t} \big)^{1/N}, \\
    & Z_{disc,t}^i = e^{-(r+\lambda a)t}Z^i_t, 
    & \olsi{Z_{disc,t}} = \big(\prod_{i=1}^N Z^i_{disc,t} \big)^{1/N}, \\
    \end{aligned}
\end{equation*}
and
\begin{equation*}
 U_i(z)  = \begin{cases}  
        \frac{1}{1- \frac{1}{\delta_i}} z^{1- \frac{1}{\delta_i}}, &  \delta_i \neq 1, \\
        \log(z), &  \delta_i = 1 .\\
    \end{cases}\\
\end{equation*}
Here $Z_{disc,t}^i$ is the tax-adjusted wealth for player $i$ discounted according to the tax-adjusted risk-free rate. $C_{disc,t}^i$ is the discounted total consumption over $[t,t+\mathrm{d} t]$ for player $i$. 

In essence, investor $i$ determines her total utility by summing up her utilities of consumption over intervals of length $\ud t$. For example, over the time interval  $[t,t+\mathrm{d} t]$, the utility of player $i$ is taken by comparing her own discounted consumption compared to that of her peers and weighted by the amount $\ud t$. In the final stage, the utility is taken by comparing her discounted terminal wealth to that of her peers and is weighted $\epsilon_i$. Thus $\epsilon_i$ represents player $i$'s preference for wealth as compared to consumption.

\section{Proofs of Propositions}\label{appendix_proof}
\subsection{Proof of Proposition \ref{thmCARAnoconsump}}\label{CARA utility proof no consump}

For convenience, we restate the dynamics \eqref{XdyanmicsMertonNoCons}, which is 
\begin{equation*}
\begin{aligned}
 & \ud X^i_t =  \left((\mu_1 - r) \pi^i_t X^i_t  + rX^i_t -\mu_2 Y^i_t  \right) \ud t + \sigma \pi^i_t X^i_t \ud W_t, \ t \in (0,T], \\
 & X^i_t = \zeta^i(t), \ t \in (-\infty,0], \\
\end{aligned}
\end{equation*}
where the delay variable, $Y^i_t$, is defined as $
Y^i_t = \int_{-\infty}^t \lambda e^{-\lambda(t- s)} X^i_{s} \ud s.$
Motivated by the analysis in \cite{delayedproblem}, we note that differentiating $Y^i_t$ gives the following differential relation
\begin{equation*}
\begin{aligned}
 & \ud Y^i_t = \lambda (X^i_t - Y^i_t) \ud t.\\
\end{aligned}
\end{equation*}
Multiplying $\ud Y^i_t$ by $a$ and adding with $\ud X^i_t$, we get
\begin{equation*}
  \ud (X^i_t + a Y^i_t) =  \left( (\mu_1 - r) \pi^i_t X^i_t  + (r + \lambda a)X^i_t + (-\mu_2 -\lambda a) Y^i_t  \right) \ud t + \sigma \pi^i_t X^i_t \ud W_t. \\
\end{equation*}
Now, using the parameter value $a = \frac{-(r+\lambda) + \sqrt{(r+\lambda)^2 - 4 \lambda \mu_2}}{2\lambda}$, we have that $a(r+\lambda a) = -\mu_2 - \lambda a$ as one can see $a$ is a root of this quadratic equation. Thus denoting $Z^i_t = X^i_t + a Y^i_t$, we get
\begin{equation*}
  \ud Z^i_t =  \left( (\mu_1 - r) \pi^i_t X^i_t  + (r + \lambda a)Z^i_t  \right) \ud t + \sigma \pi^i_t X^i_t \ud W_t. \\
\end{equation*}
Multiplying by the integrating factor $e^{-(r+\lambda a)t}$ and denoting $Z_{disc,t}^i = e^{-(r+\lambda a)t}Z^i_t$, we get
\begin{equation}\label{Z disc dynamics}
  \ud Z_{disc,t}^i =  (\mu_1 - r) \pi^i_t e^{-(r+\lambda a)t} X^i_t \ud t + \sigma \pi^i_t e^{-(r+\lambda a)t} X^i_t \ud W_t. \\
\end{equation}
Now, for each $i$, define the transformed control $\tilde{\pi}^i$ by
\begin{equation}\label{Prop 4.1 transformed control}
\begin{aligned}
\tilde{\pi}^i_t = \pi^i_t X^i_t e^{-(r+\lambda a)t}.\\
\end{aligned}
\end{equation}
Substituting the transformed controls, we have the following controlled SDE for each $i \in \{1, \cdots, N \}$
\begin{equation}\label{prop 4.1 transformed eq 1}
  \ud Z_{disc,t}^i =   (\mu_1 - r) \tilde{\pi}^i_t  \ud t + \sigma \tilde{\pi}^i_t  \ud W_t, \quad t \in (0,T]. \\
\end{equation}
Restating the reward function defined in Proposition \ref{thmCARAnoconsump}, one has that the reward for each player $i \in \{ 1, \cdots, N \}$ is given by
\begin{equation}\label{prop 4.1 transformed eq 2}
\begin{aligned}
    J^i[\tilde{\bm \pi}] = \mathbb{E} \left[U_i \left( Z^i_{disc,T} - \theta_i \olsi{Z}_{disc,T}  \right) \right], \\
\end{aligned}
\end{equation}
where
\begin{equation}\label{prop 4.1 transformed eq 3}
\begin{aligned}
    \olsi{Z}_{disc,T}  = \frac{1}{N} \sum_{i=1}^N Z^i_{disc,T}, \quad   U_i(z)  = - \exp(-\frac{1}{\delta_i} z).\\
\end{aligned}
\end{equation}
 Therefore, the characterization of controls for the Nash equilibrium can be considered through Eqs.~\eqref{prop 4.1 transformed eq 1}--\eqref{prop 4.1 transformed eq 3}. This is precisely the same problem described in \cite[Section~2, Corollary~4]{lackerNoConsump}, which gives a closed-loop Nash equilibrium given by the controls
\begin{equation*}
  \tilde{\pi}^{i,*}_t 
     = \frac{\mu_1-r}{\sigma^2} \left( \delta_i + \frac{\theta_i\bar{\delta}}{1-\bar{\theta}} \right),
\end{equation*}
for $i \in \{1 ,\cdots, N\}$, where $\bar{\delta} = \frac{1}{N} \sum_{i=1}^N \delta_i$ and $\bar{\theta} = \frac{1}{N} \sum_{i=1}^N \theta_i$. Substituting $\tilde{\pi}^{i,*}_t$ into Eq.~\eqref{Prop 4.1 transformed control}, we get that there is a closed-loop Nash equilibrium control given by the choices of each player $i$ by
\begin{equation*}
  \pi^{i,*}_t X_t^{i, \ast}
     = \frac{\mu_1-r}{\sigma^2} \left( \delta_i + \frac{\theta_i\bar{\delta}}{1-\bar{\theta}} \right) \frac{1}{e^{-(r+\lambda a)t}},
\end{equation*}
That is, at time $t$, the equilibrium strategy is to invest a deterministic dollar amount into the risky asset, independent of the current wealth level. This proves Proposition \ref{thmCARAnoconsump}.

%\dbchange{when $X^i_t \neq 0$. Otherwise, when $X^i_t =0$, we will understand that the $\pi^{i,*}_t X^i_t$ terms are replaced by  $\frac{\mu_1-r}{\sigma^2} \left( \delta_i + \frac{\theta_i\bar{\delta}}{1-\bar{\theta}} \right) \frac{1}{e^{-(r+\lambda a)t}}$ in the dynamics \eqref{XdyanmicsMertonNoCons}.} 

\subsection{Proof of Proposition \ref{thmCRRAnoconsump}}\label{CRRA utility proof no consump}
Since the dynamical system is exactly the same as in Appendix \ref{CARA utility proof no consump}, we know that $Z^i_t = X^i_t + aY^i_t$ satisfies
\begin{equation}\label{z eqn pre transformation CRRA no consump}
  \ud Z^i_t =  \left( (\mu_1 - r) \pi^i_t X^i_t  + (r + \lambda a)Z^i_t  \right) \ud t + \sigma \pi^i_t X^i_t \ud W_t. \\
\end{equation}
We are assuming that the initial path $X^i_t = \zeta^i(t)$ for $t \leq 0$ is chosen so that $Z^i_0 = X^i_0 + aY^i_0 > 0$. 

Next, since $\pi^i$ is admissible, it has the property $|\pi^i_t X^i_t| \leq K |Z^i_t|$ for some $K>0$. In particular, this means that $\pi^i_t X^i_t = 0$ whenever $Z^i_t = 0$. We will soon show that $Z^i_t >0$. 
Now, for each $i$, we define the transformed control $\tilde{\pi}^i$ by
\begin{equation}\label{transformed control (before proof) Prop 3.2 1}
\tilde{\pi}^i_t = \begin{cases}
        \pi^i_t \frac{X^i_t}{Z^i_t}, &  Z^i_t \neq 0, \\
        0, & Z^i_t = 0.
    \end{cases}
\end{equation}
Since $|\pi^i_t X^i_t| \leq K |Z^i_t|$, we have $|\tilde{\pi}^i_t| \leq K$ and claim that the dynamics \eqref{z eqn pre transformation CRRA no consump} can be written in terms of $\tilde{\pi}^i_t$ by
\begin{equation}\label{z eqn post transformation CRRA no consump}
  \ud Z^i_t =  \left( (\mu_1 - r) \tilde{\pi}^i_t Z^i_t  + (r + \lambda a)Z^i_t  \right) \ud t + \sigma \tilde{\pi}^i_t Z^i_t\ud W_t. \\
\end{equation}
%This holds as Eq.~\eqref{z eqn post transformation CRRA no consump} is consistent with Eq.~\eqref{z eqn pre transformation CRRA no consump} for both $Z^i_t \neq 0$ as well as $Z^i_t = 0$, since $\pi^i_t X^i_t = 0$ whenever $Z^i_t = 0$.
Since $|\tilde{\pi}^i_t| \leq K$, we have that $\mathbb{E} \left[ \int_0^t \left(\tilde{\pi}^i_s \right)^2 \ud s \right] < \infty$ for each $t \in (0,T]$ and therefore, $Z^i_t$ has the unique solution 
\begin{equation*}
    Z^i_t = Z^i_0 \exp\left[ (r+\lambda a)t + \int_0^t[ (\mu_1 - r) \tilde{\pi}^i_s - \frac{1}{2}\sigma^2 (\tilde{\pi}^i_s)^2] \ud s + \int_0^t \tilde{\pi}^i_s \ud W_s  \right],
\end{equation*}
for all $t \in [0,T]$. This indicates $Z^i_t > 0$ as we have assumed the initial path $X^i_{(-\infty,0]}$ is such that $Z^i_0 = X^i_0 + aY^i_0 > 0$. Also, one can show $X^i_t >0$ for all $t \leq T$. First $X^i_t = \zeta^i(t) > 0$ for $t \leq 0$. Now, by contradiction take $t' > 0$ to be the first hitting time $X^i_{t'} = 0$. Since $X^i_{t} > 0$ for $t<t'$, we see that $Y^i_{t'} = \int_{-\infty}^{t'} \lambda e^{-\lambda(t- s)} X^i_{s} \ud s > 0$. Therefore $0 < Z_{t'}^i = X^i_{t'} + aY^i_{t'} = aY^i_{t'}$, which is a contradiction as $a<0$ and $Y^i_{t'} > 0$. Since $Z^i_t >0$, the transformed control from Eq.~\eqref{transformed control (before proof) Prop 3.2 1} can be written simply as 
\begin{equation}\label{transformed control (after proof) Prop 3.2 1}
\tilde{\pi}^i_t = \pi^i_t \frac{X^i_t}{Z^i_t},
\end{equation}
and since $X^i_t > 0$ as well, this transformation is invertible which we will use later.

Now, multiplying Eq.~\eqref{z eqn post transformation CRRA no consump} by the integrating factor $e^{-(r+\lambda a)t}$ , we can write the equation for $Z^i_{disc,t} = e^{-(r+\lambda a)t} Z^i_t$ as
\begin{equation}\label{Prop 3.2 reduced eq 1}
  \ud Z_{disc,t}^i =   (\mu_1 - r) \tilde{\pi}^i_t Z_{disc,t}^i \ud t + \sigma \tilde{\pi}^i_t Z_{disc,t}^i  \ud W_t, \quad t \in (0,T]. \\
\end{equation}
%This gives the result that $Z_{disc,t}^i = Z_{disc,0}^i \exp \{ \int_0^t[ (\mu_1 - r) \tilde{\pi}^i_s - \frac{1}{2}\sigma^2 (\tilde{\pi}^i_s)^2] \ud s + \int_0^t \tilde{\pi}^i_s \ud W_s \}$. Therefore, $Z_{disc,t}^i > 0$ almost surely, and the utility is well defined.
We restate the reward function in Proposition \ref{thmCARAnoconsump} for convenience. We have that the reward for player $i$ is given by  
\begin{equation}\label{Prop 3.2 reduced eq 2}
\begin{aligned}
    J^i & = \mathbb{E} \left[U_i \left( Z^i_{disc,T}  \olsi{Z}_{disc,T}^{-\theta_i}  \right) \right], 
\end{aligned}
\end{equation}
where $0 < \theta_i < 1$ for each $i$, and where
\begin{equation}\label{Prop 3.2 reduced eq 3}
  \olsi{Z}_{disc,T}  =  \left(\prod_{i=1}^N Z^i_{disc,T}  \right)^{1/N}, \qquad  U_i(z) = \begin{cases}  
        \frac{1}{1- \frac{1}{\delta_i}} z^{1- \frac{1}{\delta_i}}, &  \delta_i \neq 1, \\
        \log(z), &  \delta_i = 1,
    \end{cases}
\end{equation}
with $\delta_i>0$. The Nash equilibrium problem defined by Eqs.~\eqref{Prop 3.2 reduced eq 1}--\eqref{Prop 3.2 reduced eq 3} for generic, progressively measurable controls $(\tilde{\pi}^i)_{i=1}^N$ such that $\mathbb{E} \left[ \int_0^t \left(\tilde{\pi}^i_s \right)^2 \ud s \right] < \infty$  falls into the formulation of \cite[Section~3, Corollary~15]{lackerNoConsump}. Using the results therein, one has a closed-loop Nash equilibrium given by the controls
\begin{equation*}
  \tilde{\pi}^{i,*}_t 
     = \frac{\mu_1-r}{\sigma^2} \left( \delta_i - \frac{\theta_i(\delta_i-1)\bar{\delta}}{1+\olsi{\theta(\delta-1)}} \right),
\end{equation*}
for $i \in \{1 ,\cdots, N\}$. Solving for $\pi^{i,*}_t$, from Eq.~\eqref{transformed control (after proof) Prop 3.2 1} we get 
\begin{equation*}
        \pi^{i,*}_t  =  \frac{\mu_1-r}{\sigma^2} \left( \delta_i - \frac{\theta_i(\delta_i-1)\bar{\delta}}{1+\olsi{\theta(\delta-1)}} \right) \frac{X^i_t + aY^i_t}{X^i_t},
\end{equation*}
which holds as we have shown that $X^i_t >0$. Lastly, we see that the Nash equilibrium controls do satisfy the admissibility condition as $|\pi^{i,*}_t  X^i_t| = \left|\frac{\mu_1-r}{\sigma^2} \left( \delta_i - \frac{\theta_i(\delta_i-1)\bar{\delta}}{1+\olsi{\theta(\delta-1)}} \right)\right| |Z^i_t|$. This completes the proof of Proposition \ref{thmCRRAnoconsump}.

\subsection{Proof of Proposition \ref{thmCRRAconsump}}\label{CRRA utility proof w consump}
We now form the solution for the problem defined by  Eq.~\eqref{XdyanmicsMertonwCons}--\eqref{costMertonwConsCRRA3}, proceeding similarly to the proof in Appendix \ref{CRRA utility proof no consump}. Restating the dynamics given in Proposition \ref{thmCRRAconsump}, we have for each $i \in \{1, \cdots, N \},$
\begin{equation*}
\begin{aligned}
 & \ud X^i_t =  \left((\mu_1 - r) \pi^i_t X^i_t  + rX^i_t -\mu_2 Y^i_t - c^i_t X^i_t \right) \ud t + \sigma \pi^i_t X^i_t \ud W_t, \ t \in (0,T], \\
 & X^i_t = \zeta^i(t), \ t \in (-\infty,0], \\
\end{aligned}
\end{equation*}
where the delay variable, $Y^i_t$, is given by $
Y^i_t = \int_{-\infty}^t \lambda e^{-\lambda(t- s)} X^i_{s} \ud s$. As usual, define $Z^i_t = X^i_t + aY^i_t$ and $Z^i_{disc,t} = e^{-(r+\lambda a)t} Z^i_t$, where $a = \frac{-(r+\lambda) + \sqrt{(r+\lambda)^2 - 4 \lambda \mu_2}}{2\lambda}$. Following the similar computation in the proof in Appendix \ref{CARA utility proof no consump}, one can show
\begin{equation}\label{transformed eqn but no transformed controls 3.3 proof}
  \ud Z_t^i =  \left( (r+\lambda a)Z^i_t + (\mu_1 - r) \pi^i_t  X^i_t - c^i_t X^i_t  \right) \ud t + \sigma \pi^i_t X^i_t \ud W_t. \\
\end{equation}
%Here we see that $Z_{disc,t}^i > 0$ as we are assuming that the initial path of $X^i_t = \zeta^i(t) > 0$ for $t \leq 0$ is such that $0 < X^i_0 + aY^i_0 = Z_{disc,0}^i$. Then since $\pi$ and $\c$ are admissible, we will have that 
Now, for each $i$, define the transformed controls $(\tilde{\pi}^i$, $\tilde{c}^i)$ by
\begin{equation*}
(\tilde{\pi}^i_t,\tilde{c}^i_t) = \begin{cases}
    (\pi^i_t \frac{X^i_t}{Z^i_t},c^i_t \frac{X^i_t}{Z^i_t}), & Z^i_t \neq 0, \\
    (0,0), & Z^i_t = 0.
\end{cases} \quad 
\end{equation*}
%Now since the admissibility conditions $|\pi^i_t X^i_t| \leq K|Z^i_t|$ and $|c^i_t X^i_t| \leq K|Z^i_t|$ give that $\pi^i_t X^i_t = c^i_t X^i_t = 0$ whenever $Z^i_t = 0$, 
Then, Eq.~\eqref{transformed eqn but no transformed controls 3.3 proof} can be written as 
\begin{equation}\label{transformed eqn 3.3 proof}
  \ud Z_t^i =  \left( (r+\lambda a)Z^i_t + (\mu_1 - r) \tilde{\pi}^i_t  Z^i_t - \tilde{c}^i_t Z^i_t  \right) \ud t + \sigma \tilde{\pi}^i_t Z^i_t \ud W_t. \\
\end{equation}
Since $|\tilde{\pi}^i_t|, |\tilde{c}^i_t| \leq K$, we have that $\mathbb{E} \left[ \int_0^t \left(\tilde{\pi}^i_s\right)^2 + \left(\tilde{c}^i_s \right)^2 \ud s \right] < \infty$ for each $t \in (0,T]$, and $Z^i_t$ has the unique solution 
\begin{equation*}
    Z^i_t = Z^i_0 \exp\left[ (r+\lambda a)t + \int_0^t[ (\mu_1 - r) \tilde{\pi}^i_s - \frac{1}{2}\sigma^2 (\tilde{\pi}^i_s)^2 - \tilde{c}^i_s] \ud s + \int_0^t \tilde{\pi}^i_s \ud W_s  \right].
\end{equation*}
Therefore, $Z^i_t >0$ and  $X^i_t >0$, following the same argument as in Appendix \ref{CRRA utility proof no consump}, and the transformed controls become simply 
\begin{equation}\label{Prop 3.3 transformed controls}
(\tilde{\pi}^i_t,\tilde{c}^i_t) =
    (\pi^i_t \frac{X^i_t}{Z^i_t},c^i_t \frac{X^i_t}{Z^i_t}),
\end{equation}
which can be inverted for a given choice of controls $(\tilde{\pi}^i_t,\tilde{c}^i_t)$. 

Next, multiplying Eq.~\eqref{transformed eqn 3.3 proof} by the integrating factor $e^{-(r+\lambda a)t}$, we get
\begin{equation}\label{eq Prop 3.3 1}
  \ud Z_{disc,t}^i =  \left( (\mu_1 - r) \tilde{\pi}^i_t Z_{disc,t}^i - \tilde{c}^i_t Z_{disc,t}^i  \right) \ud t + \sigma \tilde{\pi}^i_t Z_{disc,t}^i \ud W_t, \quad t \in (0,T]. \\
\end{equation}
%Solving the above SDE gives $Z_{disc,t}^i = Z_{disc,0}^i \exp \{ \int_0^t[ (\mu_1 - r) \tilde{\pi}^i_s - \tilde{c}^i_s - \frac{1}{2}\sigma^2 (\tilde{\pi}^i_s)^2] \ud s + \int_0^t \tilde{\pi}^i_s \ud W_s \}$, \db{To state this as solution, I need to show the admissible set of $\pi^i$, $c^i$ implies that the integrals of $\tilde{\pi}^i_s$, $\tilde{c}^i_s$ are a.s. finite....} showing that $Z_{disc,t}^i > 0$ as we assume $Z_{disc,0}^i = X^i_0 + aY^i_0 > 0$ because of the choice of the initial path $X^i_t = \zeta^i(t)$ for $t \leq 0$. Next, since $Z_{disc,t}^i > 0$, we can show that $X^i_t > 0$. This holds as by contradiction we can take $t' > 0$ to be the first hitting time $X^i_{t'} = 0$. Since $X^i_{t} > 0$ for $t<t'$, we see that by the integral definition  $Y^i_{t'} > 0$. But $Z_{disc,t}^i > 0$ for all $t$, so $0 < Z_{t'}^i = X^i_{t'} + aY^i_{t'} = aY^i_{t'} $. This is a contradiction as $a<0$ and $Y^i_{t'} > 0$. Therefore, both $Z_{disc,t}^i > 0$ as well as $X^i_t > 0$ a.s. for all $t>0$. This means that the reward function defined in Proposition \ref{thmCRRAconsump} is well defined.
For convenience, we restate the reward defined in Proposition \ref{thmCRRAconsump},  which is
\begin{equation*}
  J^i = E\Big[\int_0^T U^i(C_{disc,t}^i \olsi{C_{disc,t} }^{-\theta_i}) \ud t + \epsilon_i U_i\Big( Z_{disc,T}^i \olsi{Z_{disc,T}}^{-\theta_i} \Big) \Big],
\end{equation*}
where
\begin{equation*}
    \begin{aligned}
   &  C_{disc,t}^i =  e^{-(r+\lambda a)t} c^i_t X^i_t, & \olsi{C_{disc,t}} = \big(\prod_{i=1}^N C^i_{disc,t} \big)^{1/N}, \\
    & Z_{disc,t}^i = e^{-(r+\lambda a)t}Z^i_t, 
    & \olsi{Z_{disc,t}} = \big(\prod_{i=1}^N Z^i_{disc,t} \big)^{1/N}, \\
    \end{aligned}
\end{equation*}
 and where $0<\theta_i<1$, $\epsilon_i > 0$, and 
\begin{equation}\label{eq Prop 3.3 2}
 U_i(z)  = \begin{cases}  
        \frac{1}{1- \frac{1}{\delta_i}} z^{1- \frac{1}{\delta_i}}, &  \delta_i \neq 1, \\
        \log(z), &  \delta_i = 1,\\
    \end{cases}\\
\end{equation}
with $\delta_i>0$.
In terms of the transformed controls defined by Eq.~\eqref{Prop 3.3 transformed controls}, we can write the reward function for player $i$ as
\begin{equation}\label{eq Prop 3.3 3}
  J^i = E\Big[\int_0^T U^i(\tilde{c}^i_t Z_{disc,t}^i \olsi{(\tilde{c}_t Z_{disc,t})}^{-\theta_i}) \ud t + \epsilon_i U_i\Big( Z_{disc,T}^i \olsi{Z_{disc,T}}^{-\theta_i} \Big) \Big],
\end{equation}
where $\overline{\tilde{c}_t Z_{disc,t}} = \big(\prod_{i=1}^N \tilde{c}^i_t Z^i_{disc,t} \big)^{1/N}$. 

%We see that any $(\pi^i,c^i) \in \mathbb{A}^i$ are transformed to $(\tilde{\pi}^i,\tilde{c}^i)$ by Eq.~\eqref{Prop 3.3 transformed controls}. We see that since $X^i_t, Z^i_t > 0$ and since $c^i$ maps to $\mathbb{R}^+$, it holds that $\tilde{c}^i_t \geq 0$. Also, we see that $|\tilde{\pi}^i|,|\tilde{c}^i| \leq K$ and so $\left[ \int_0^t \left(\tilde{\pi}^i_s\right)^2 + \left(\tilde{c}^i_s \right)^2 \ud s \right] < \infty$. 

The Nash equilibrium problem defined by Eqs.~\eqref{eq Prop 3.3 1}--\eqref{eq Prop 3.3 3} is given in \cite[Corollary~2.3]{lackerwconsump} for the generic, progressively measurable controls $(\tilde{\pi}^i,\tilde{c}^i)_{i=1}^N$ that satisfy the admissibility conditions $\mathbb{E} \left[ \int_0^t \left(\tilde{\pi}^i_s\right)^2 + \left(\tilde{c}^i_s \right)^2 \ud s \right] < \infty$ and $\tilde{c}^i_t \geq 0$ for each $t \in [0,T]$. Specifically, \cite[Corollary~2.3]{lackerwconsump} gives that there exists a closed-loop Nash equilibrium given by the controls
\begin{equation*}
  \begin{aligned}
        \tilde{\pi}^{i,*}_t & =  \frac{\mu_1-r}{\sigma^2} \left( \delta_i - \frac{\theta_i(\delta_i-1)\bar{\delta}}{1+\olsi{\theta(\delta-1)}} \right), \\
        \tilde{c}^{i,*}_t & = \begin{cases}
            \left(\beta_i^{-1} + (\gamma_i^{-1} - \beta_i^{-1}) e^{-\beta_i(T-t} \right)^{-1} , & \delta_i \neq 1, \\
            \left(T-t +\gamma_i^{-1} \right)^{-1}, & \delta_i = 1,\\   
        \end{cases} \\
    \end{aligned}  
\end{equation*}
for $i \in \{1 ,\cdots, N\}$, where the notations $\bar{\delta}, \olsi{\theta(\delta-1)}$ correspond to the arithmetic average, and the parameters $\beta_i$ and $\gamma_i$ are given by
\begin{equation*}
  \begin{aligned}
    \beta_i & = \frac{1}{2}(1-\delta_i) \left( \frac{\mu_1-r}{\sigma}\right)^2 \left(1- \frac{\theta_i \bar{\delta}}{1+\olsi{\theta(\delta-1)}} \right) \left( \delta_i - \frac{\theta_i \bar{\delta}}{1+\olsi{\theta(\delta-1)}} (\delta_i - 1)\right),\\
    \gamma_i & = \epsilon_i^{-\delta_i} \left( \left( \prod_{k=1}^N  \epsilon_k^{\delta_k} \right)^{1/N} \right)^{\theta_i(\delta_i-1)/(1+\olsi{\theta(\delta-1)})}.\\
    \end{aligned}  
\end{equation*}
Substituting back for $(\pi^{i,*}, c^{i,*})$ through the transformations defined by Eq.~\eqref{Prop 3.3 transformed controls}, we see that  $(\pi^{i,*}, c^{i,*})$  is indeed in $\mathbb{A}^i$ and the result from Proposition \ref{thmCRRAconsump} holds.
\end{document}